\documentclass[final,3p]{elsarticle}
\usepackage{subcaption}
\usepackage[colorlinks=true,breaklinks=true,pdftex]{hyperref}

\journal{GMP 2026}

\biboptions{authoryear}

	\usepackage{amsmath,amsthm,amssymb}
	\usepackage{graphicx}
	\usepackage{bm}
	\usepackage{array}
	\usepackage{booktabs}
	\usepackage{multirow}
	\usepackage{siunitx}
	\usepackage{caption}
	\usepackage{float}
	\usepackage{mathdots}
	\usepackage{color}

	\def\RR{\mathbb{R}}
	\def\NN{\mathbb{N}}

	\newcommand{\as}{\mathrel{:=}}
	\newcommand{\sa}{\mathrel{=:}}
	
	\newcommand{\intbox}{%
		{\!\;\fboxsep0pt\framebox{\phantom{$I$\!\!\,}}\!\;}}
	
	\newcommand{\bfm}{\bm{m}}
	
	\newcommand{\bfr}{\bm{r}}
	
	\newcommand{\bfx}{\bm{x}}
	\newcommand{\bfxi}{\bm{\xi}}
	
	\newcommand{\bemph}[1]{\emph{\bfseries#1}}
	
	\newcommand{\iRR}{\intbox\RR}
	\newcommand{\intboxTaylor}{\intbox^{\mathrm{T}}}
	\newcommand{\intboxLagrange}{\intbox^{\mathrm{L}}}
	\newcommand{\intboxHermite}{\intbox^{\mathrm{H}}}
	
	\newcommand{\OmegaLagrange}{\Omega_{\mathrm{L}}}
	\newcommand{\errorLagrange}{E^{\mathrm{L}}}
	\newcommand{\TaylorLagrange}{T^{\mathrm{L}}}
	\newcommand{\remainderLagrange}{R^{\mathrm{L}}}
	
	\newcommand{\OmegaHermite}{\Omega_{\mathrm{H}}}
	\newcommand{\errorHermite}{E^{\mathrm{H}}}
	\newcommand{\TaylorHermite}{T^{\mathrm{H}}}
	\newcommand{\remainderHermite}{R^{\mathrm{H}}}
	
	\newcommand{\Delannoy}[2]{\begin{bmatrix}#1\\#2\end{bmatrix}}
	\newcommand{\textDelannoy}[2]{%
		\bigl[\!\begin{smallmatrix}#1\\#2\end{smallmatrix}\!\bigr]}

	\DeclareMathOperator{\Hausdorff}{\mathit{d}_{\mathrm{H}}}
	
	\providecommand{\abs}[1]{{\lvert{#1}\rvert}}

	\providecommand{\floor}[1]{{\lfloor{#1}\rfloor}}
    
    \definecolor{mygreen}{RGB}{0,191,0}

	\providecommand{\xKH}[1]{} 
	\providecommand{\xCY}[1]{} 
	\providecommand{\xBW}[1]{} 

    \newcommand{\rev}[1]{{\color{red}#1}}

	\newcommand{\ignore}[1]{}
	
	\newcommand{\totalTime}{\mathrm{TotalTime}}
	\newcommand{\totalWidth}{\mathrm{TotalWidth}}
	\newcommand{\cored}[1]{\textcolor{mygreen}{#1}}
	\theoremstyle{plain}
	\newtheorem{theorem}{Theorem}

	\theoremstyle{definition}

\begin{document}

\begin{frontmatter}

\title{Bivariate range functions with superior convergence order\footnote{ To appear in GMP 2026 Proceedings.  This arXiv version has additional experimental data in Appendix C. } }
\author[nyu]{Bingwei Zhang}
\author[nyu]{Thomas Chen}
\author[usi]{Kai Hormann}
\author[nyu]{Chee Yap}
\address[nyu]{Courant Institute, New York University}
\address[usi]{Facolt\`{a} di Scienze Informatiche, Universit\`{a} della Svizzera italiana}

\begin{abstract}
	Range functions are a fundamental tool for certified computations in
	geometric modeling, computer graphics, and robotics,
	but traditional range functions have only quadratic convergence order
	($m=2$).  For ``superior'' convergence order (i.e., $m>2$), we exploit the
	Cornelius--Lohner framework in order to
	introduce new bivariate range functions based
	on Taylor, Lagrange, and Hermite interpolation.
	In particular, we focus on practical
	range functions with cubic and quartic convergence order.
	We implemented them in Julia and provide experimental
	validation of their performance in terms
	of efficiency and efficacy.

\ignore{ Following is more appropriate for intro? 
	They provide
	mathematically rigorous bounds on the range of a function over any
	specified domain, which are crucial for developing robust algorithms
	for collision detection, intersections tests, and many other tasks
	that require safe decisions. The simplest approach is to evaluate a
	given function with interval arithmetic, resulting in the natural
	interval extension of the function, but tighter enclosures of the
	exact range can be computed with centered forms. While the latter
	exhibit quadratic convergence order, the goal of this paper is to
	derive novel bivariate range functions with higher convergence orders.
	}
\end{abstract}


\end{frontmatter}



\section{Introduction}
	Many problems in geometric modeling, computer graphics,
	and robotics, such as surface-surface intersection,
	constructive solid geometry (CSG) operations,
	ray-tracing, and
	collision detection, can be solved robustly
	using interval analysis and range (or inclusion) functions
	\citep{Bowyer:2000:IMI, Snyder:1992:IAF, Duff:1992:IAA}.
	The core idea of range
	functions is to calculate an interval that is guaranteed
	to enclose the range of a function $f$ over any given domain $B$,
	where $B$ is typically a box, simplex, or ball in Euclidean space
	\citep{Ratschek:1984:CMF}.  
	Range functions provide various certified predicates for finding
	roots of $f$.  One such predicate is the exclusion predicate of root
	finding: if the estimated range of $f$ over $B$ does \emph{not} contain zero,
	then $f$ has no roots in $B$.  
    With good range
	functions, provably near-optimal complexity bounds can be achieved
	for isolating real and complex roots
	\citep{Sharma:2012:NOT,Becker:2018:ANS}.  Combined with subdivision,
	range functions lead to very efficient branch-and-bound algorithms
	for the problems mentioned before \citep{Martin:2002:COI,Sichetti:2025:MAD}.
    In particular, \emph{bivariate} range functions have been used 
    for the isotopic approximation of curves \citep{Plantinga:2004:IAO,cxy,cxyz}, 
    and they can be used for finding the roots of mappings from $\RR^2$ to $\RR^2$
    \citep{Cheng:2023:CNR}.

	Clearly, the efficiency of these algorithms depends on
	(1) the efficiency of
	evaluating the range functions and
	(2) the efficacy (or tightness) of the estimated ranges.  
	A simple measure of tightness is the order of convergence.
	The classic approach to range functions for $f$ is based on
	the Taylor expansion of $f$ to some degree $n\ge 1$.
	Unfortunately, the order of convergence
	remains quadratic ($m=2$), regardless of $n$
	\citep{Ratschek:1984:CMF}.
	\citet{Cornelius:1984:CTR} introduced the ``CL framework''
	that finally achieves
	{\em superior}
	convergence orders (i.e., $m>2$). While
	they can only reach convergence orders of up to $m\le 6$, 
	due to the underlying implicit computational model,
    \citet[Theorems B and C]{Hormann:2023:RFO} 
	showed how to achieve convergence order $m$ for any $m>1$
	by generalizing the CL framework.

	
	Superior convergence orders have an efficiency cost.
	So in any application, one must choose a ``sweet spot''
	convergence order that achieves practical trade-offs
	among efficiency, efficacy, and ease of implementation. 
	Typically, the sweet spot is some small $m$, say $m<6$. 
	E.g., for the EVAL root isolation algorithm 
		\citep{Hormann:2021:NRF},
	such considerations led to the sweet spot order of $m=3$.
	Thus, in our current
	implementations, we focus on convergence orders up to $m\le 4$.
	Although our techniques are in principle able to reach
	larger values of $m$, it is clear from our development
	below that their implementation will be much more involved.

	In this paper, we focus on the bivariate setting
	and derive new range functions that advance the state of the art
	in terms of convergence order.
	Specifically, we introduce bivariate range functions with
	cubic and quartic convergence orders that are
	based on Taylor, Lagrange, and Hermite interpolation.
	The geometric modeling literature has a
	well-known ``Bernstein range'' 
	for any polynomial $f$ over a box $B$, namely the interval
	between the minimum and maximum Bernstein coefficients over $B$
	\citep{farin:bk}.
	But our experiments will not compare this
	Bernstein range function with our range functions.
	We offer two justifications for this omission.
	First, the Bernstein approach is fundamentally different from the
	interpolation approach of our methods, since it works
	only for polynomials and we know of no range functions
	based on the Bernstein approach with 
	a superior convergence order.
	Second, the comprehensive study of
	\citet{Martin:2002:COI}
	compared a set of 30 range functions for 
	tracing planar curves, and they concluded that 
	the Bernstein method performs comparably to the 
	centered form Taylor method, but the latter is easier to implement.
	This centered form Taylor method is just $\intbox^T_{2,n}$
	in our experiments.
	
\ignore{
	\subsection{Brief Literature Review}
	We noted that the CL Framework can only achieve
	a convergence order of $k\le 6$.  
	... go into the generalization of CL framework...

}

	\subsection{Range functions}
	For any bivariate real function $f\colon\RR^2\to\RR$ and any
	$S\subseteq\RR^2$, the \emph{range} of $f$ on $S$ is the set
	$f(S)\as\{f(x):x\in S\}$. With $\iRR^2$ denoting the set of closed,
	axis-aligned rectangles, also called \emph{boxes} and $\iRR$ denoting
	the set of closed intervals, we call $\intbox f\colon\iRR^2\to\iRR$ a
	\emph{range function} for $f$, if $\intbox f(B)\supseteq f(B)$ for
	all $B\in\iRR^2$. For any box $B=[a_x,b_x]\times[a_y,b_y]\in\iRR^2$,
	the \emph{midpoint}, \emph{radius}, and (maximal) \emph{width} of $B$
	are defined as $\bfm=(m_x,m_y)\as(a_x+b_x,a_y+b_y)/2$,
	$\bfr=(r_x,r_y)\as(b_x-a_x,b_y-a_y)/2$, and
	$w(B)\as2\max(r_x,r_y)$.We further define the \emph{width} and the
	\emph{magnitude} of any $I=[a,b]\in\iRR$ as $w(I)\as b-a$ and
	$\abs{I}\as\max\{\abs{a},\abs{b}\}$.

	For any $m\ge1$, we say that the range function $\intbox f$ exhibits
	\emph{order $m$ convergence} on $B_0\in\iRR^2$, if there exists a
	constant $C>0$ that may depend on $f$ and $B_0$, such that
		\begin{equation}\label{eq:order-m-convergence}
		  \Hausdorff( f(B), \intbox f(B) ) \le C {w(B)}^m
		\end{equation}
	for all $B\subset B_0$, where
	$\Hausdorff([a,b],[a',b'])\as\max\{\abs{a-a'},\abs{b-b'}\}$ denotes
	the Hausdorff distance between intervals.

	For a rational function $f$, given in terms of a specific
	\emph{expression} (or \emph{formula}) $E(x,y)$, the simplest range
	function over $B=I\times J$ with $I,J\in\iRR$ is the \emph{natural
	interval extension} \citep[Section~3.3]{Moore:1979:MAA} $\intbox
	f(B)\as E(I,J)$, obtained by replacing the variables with intervals
	and utilizing interval arithmetic during the evaluation of the
	expression $E$. If $\intbox f(B)$ is well-defined for all $B\subseteq
	B_0$, then $\intbox f$ is convergent over $B_0$
	\citep[Chapter~4]{Moore:1979:MAA}, but the convergence is only of
	order 1 in general. Quadratic convergence is instead achieved by
	\emph{centered forms}, for example, the celebrated \emph{mean-value
	form} \citep[Chapter~3]{Ratschek:1984:CMF}.

	The breakthrough to achieve range functions with
	superior convergence comes from the CL Framework
	\citep{Cornelius:1984:CTR}.
	The idea is to split $f$ into two parts
	$f(x)=g(x)+R_g(x)$ where $g$ is called the \emph{exact part}
	and $R_g\as f-g$ the \emph{remainder part},
	and for any interval $I$, we compute $g(I)$ exactly
	while $R_g(I)$ can be approximated by some range
	function $\intbox R_g(I)$:

	\begin{theorem}[\protect{\citeauthor{Cornelius:1984:CTR}, \citeyear{Cornelius:1984:CTR}, Theorem~4}]\label{theorem:CL-form}
		If the function $f\colon\RR\to\RR$ is continuous over
		$I_0\in\iRR$ and can be written as
		$f(x)=g(x)+R_g(x)$ for any $x\in I_0$,
		then the range function
			\begin{equation}\label{eq:CL-form}
				\intbox f(I) \as g(I) + \intbox R_g(x)
			\end{equation}
		satisfies
			\begin{equation*}
			  \Hausdorff( f(I), \intbox f(I) ) \le w(\intbox R_g(I))
			\end{equation*}
			for any $I\subseteq I_0$.
	\end{theorem}
	In particular, the convergence order of $\intbox f(I)$
	is that of $\intbox R_g(I)$, and we know how to
	achieve superior convergence orders for $\intbox R_g$.
	The problem with \eqref{eq:CL-form}
	is that there are no\footnote{
			\citet{Cornelius:1984:CTR}
			assumed that if we have an algebraic expression
			for $g(x)$, then we can compute its roots exactly. 
			E.g., if $g(x)$ is a polynomial of degree at most $4$,
			then we can compute its roots.  Although
			this is a common mathematical model,
			there are no practical exact implementations.
	}
	practical arithmetic models
	(standard models of Numerical Analysis, IEEE standard,
	bigNumber packages, etc.) for computing $g(I)$ exactly,
	except for very special $g$.
	E.g., $g(x)=3x^2+1$ is okay, but $g(x)=3x^2+x$ is not.
	  For that reason, \citet{Hormann:2023:RFO}
	proposed to allow for $g$ to be approximated by a range function
	$\intbox g$.

	\begin{theorem}[\protect{\citeauthor{Hormann:2023:RFO}, \citeyear{Hormann:2023:RFO}, Theorem~B}]\label{theorem:CL-form-generalized}
		If the function $f\colon\RR^2\to\RR^2$ is continuous over
		$B_0\in\iRR^2$ and can be written as
		$f(\bfx)=g(\bfx)+R_g(\bfx)$ for any $\bfx\in B_0$,
		then the range function
			\begin{equation}\label{eq:CL-form-generalized}
				\intbox f(B) \as \intbox g(B)+\intbox R_g(B)
			\end{equation}
		satisfies
			\begin{equation*}
			  \Hausdorff( f(B), \intbox f(B) )
			  		\le \Hausdorff( g(B),\intbox g(B))+w(\intbox R_g(B))
			\end{equation*}
			for any $B\subseteq B_0$.
	\end{theorem}

    Note that while \citet{Hormann:2023:RFO} proved this theorem 
    for univariate functions, their proof carries over to the
	multivariate setting without major changes and is stated here 
    for bivariate functions, as needed below.
	By~\eqref{eq:order-m-convergence}, an immediate consequence of
	Theorem~\ref{theorem:CL-form-generalized}
	is that the
	\emph{generalized Cornelius--Lohner
	form}~\eqref{eq:CL-form-generalized}
	has order-$m$ convergence, 
	provided we ensure
    \begin{equation}\label{eq:strong-range-fn}
			\Hausdorff( g(B),\intbox g(B))\le w(\intbox R_g(B)) \le C {w(B)}^m.
			\end{equation}
	For a large class of functions, it is possible to compute
	range functions $\intbox g(B)$ that satisfy
	\eqref{eq:strong-range-fn}
    \citep{Hormann:2023:RFO}.
	E.g., if $g(x)$ is a polynomial, then we need the ability
    to approximate the roots $\xi$ of the derivative
    $g'$ to any desired error $\delta>0$.  Then, for any $\varepsilon>0$,
    we can determine $\delta=\delta(\varepsilon)$
    so that $|\xi-\tilde\xi|\le\delta$ implies
    $|g(\tilde\xi)-g(\xi)|\le \varepsilon$.  The values $g(\tilde\xi)$ can then be used to determine $g(I)$ to within $\varepsilon$.
	\ignore{
	The most common choice is to use an interpolating polynomial with
	small degree $d$ for the exact part $g$, because closed-form
	expressions exist for such polynomials ($d\le5$ in the univariate and
	$d\le3$ in the bivariate setting), and it is possible to define a
	range function $\intbox R_g$ for the remainder that
	satisfies~\eqref{eq:remainder-condition} for $m=d+1$.
	}
	\subsection{Contributions and overview}
	In this paper, we propose and discuss three novel
	bivariate 
	range functions
	based on the generalized Cornelius--Lohner framework. The first
	(Section~\ref{sec:Taylor-forms}) stems from the Taylor expansion of
	$f$ about the midpoint of the interval and generalizes the
	quadratically convergent centered Taylor forms
	\citep[Section~3.5]{Ratschek:1984:CMF}. The second
	(Section~\ref{sec:Lagrange-forms}) and third
	(Section~\ref{sec:Hermite-forms}) apply the idea of Cornelius and
	Lohner recursively and are inspired by the corresponding univariate
	constructions in \citep{Hormann:2021:NRF} and
	\citep{Hormann:2023:RFO}, respectively. Our numerical experiments (Section~\ref{sec:experiments}) confirm the theoretically proven 
    convergence orders of the proposed range functions and identify 
    those with cubic convergence order ($m=3$) as the sweet spot, that 
    is, the best compromise between efficiency and efficacy.
    In the appendices, we provide exact expressions for the endpoints of $g(B)$ for low degree polynomials $g$ in one variable (\ref{app:A}) or two variables (\ref{app:B}).
	
	\xKH{add a summary of the numerical results and some conclusions that
	follow from them}

\section{Range functions based on centered Taylor expansions}
\label{sec:Taylor-forms}
	Abbreviating the higher order \emph{partial derivatives} of
	$f\colon\RR^2\to\RR$ by
	\[
	  f^{(i,j)} \as D^{(i,j)} f
	  	= \frac{\partial^{i+j}}{\partial x^i \partial y^j} f,
	\]
	and denoting the \emph{$k$-th differential} of $f$ by
	\[
	  d^k f
	    \as \biggl( h_x \frac{\partial}{\partial x}
			+ h_y \frac{\partial}{\partial y} \biggr)^k f
	     =  \sum_{j=0}^k \binom{k}{j} h_x^{k-j} h_y^j f^{(k-j,j)},
	\]
    we recall \emph{Taylor's theorem} \citep{Courant:1989:ITC}, by which
	any $(n+1)$-times continuously differentiable function $f$, where
	$n\in\NN_0$, can be expressed as the \emph{Taylor expansion}
	\[
	  f(x,y) = T_n(x,y) + R_n,
	\]
	that is, as the sum of the \emph{$n$-th degree Taylor polynomial} of
	$f$ about $(u,v)$,
	\begin{equation}\label{eq:Taylor-polynomial}
	  T_n(x,y) = f(u,v) + \sum_{k=1}^n \frac{1}{k!}
	  	d^k f(u,v) \qquad\text{with}\qquad h_x = x-u, \quad h_y = y-v,
	\end{equation}
	and the \emph{remainder}
	\[
	  R_n = \frac{1}{(n+1)!} d^{n+1} f(u + \theta h_x, v + \theta h_y)
	\]
	for some $\theta\in(0,1)$ (see Figure~\ref{fig:Taylor-interpolation}).
	
    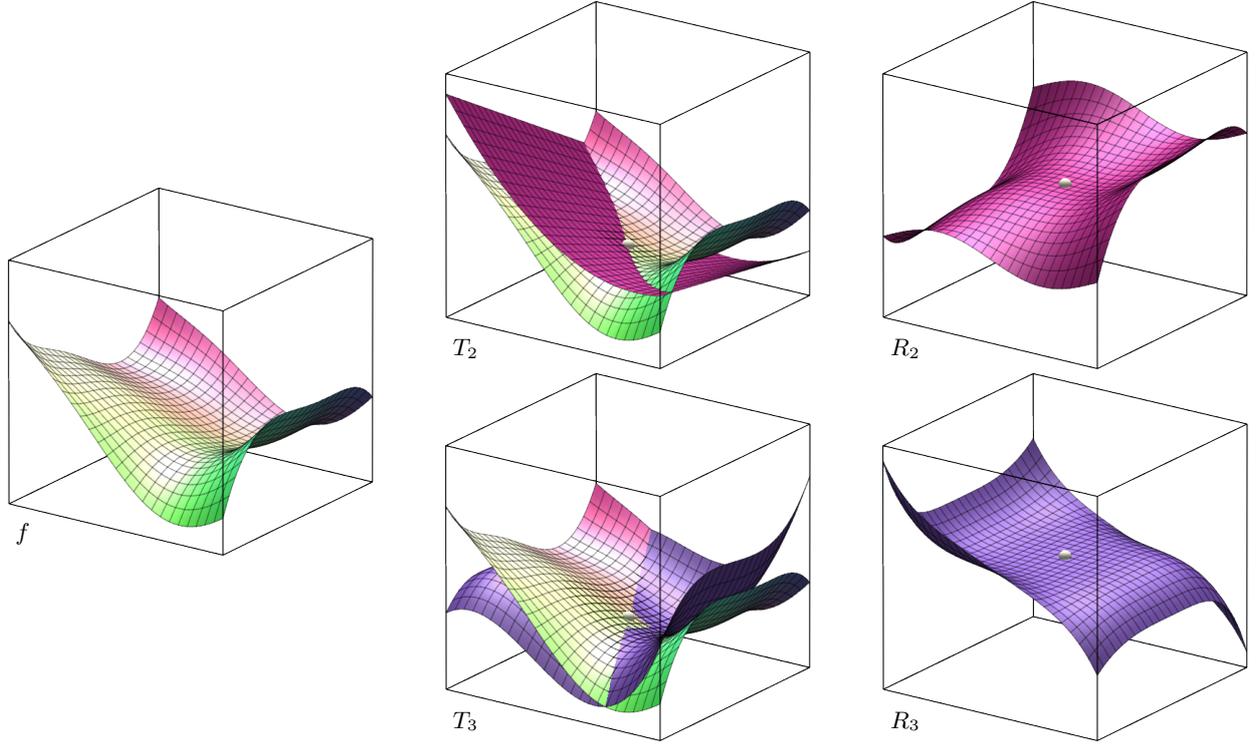
\begin{figure}[t!]\centering\small
      \setlength{\unitlength}{0.1\linewidth}
      \begin{picture}(10,6)
        \put(0,1.5){\makebox(0,0)[bl]{\includegraphics[width=3\unitlength]{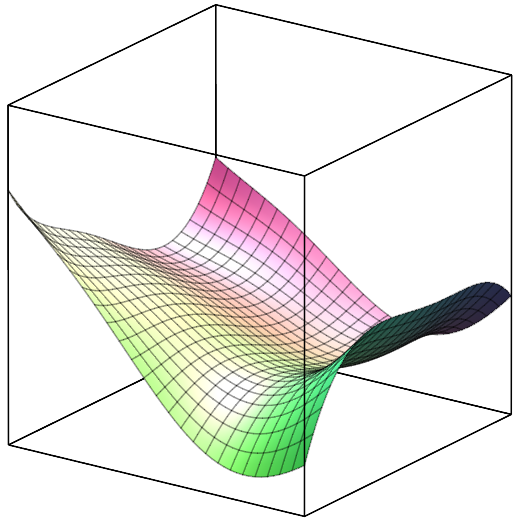}}}
        \put(0.1,1.6){\makebox(0,0)[bl]{$f$}}
        \put(3.5,3){\makebox(0,0)[bl]{\includegraphics[width=3\unitlength]{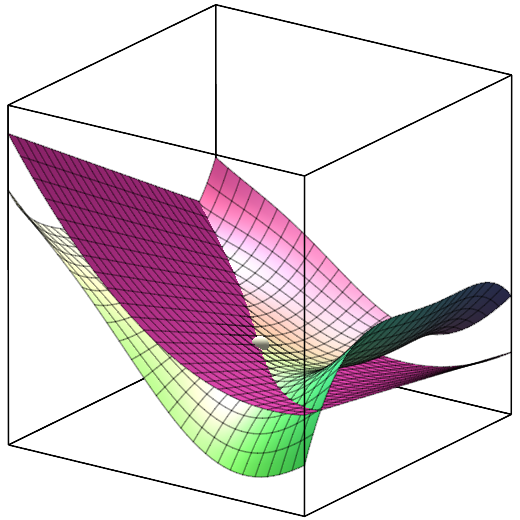}}}
        \put(3.6,3.1){\makebox(0,0)[bl]{$T_2$}}
        \put(3.5,0){\makebox(0,0)[bl]{\includegraphics[width=3\unitlength]{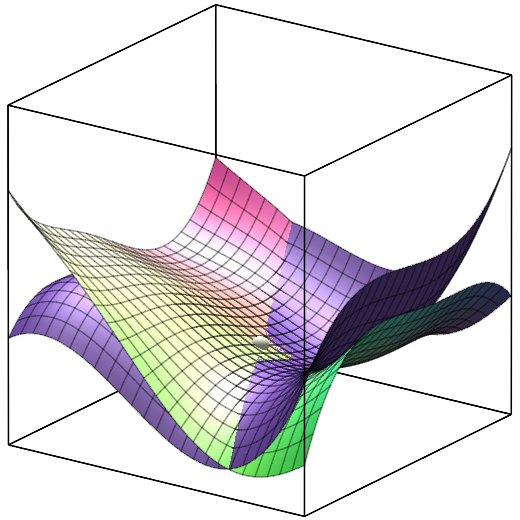}}}
        \put(3.6,0.1){\makebox(0,0)[bl]{$T_3$}}
        \put(7,3){\makebox(0,0)[bl]{\includegraphics[width=3\unitlength]{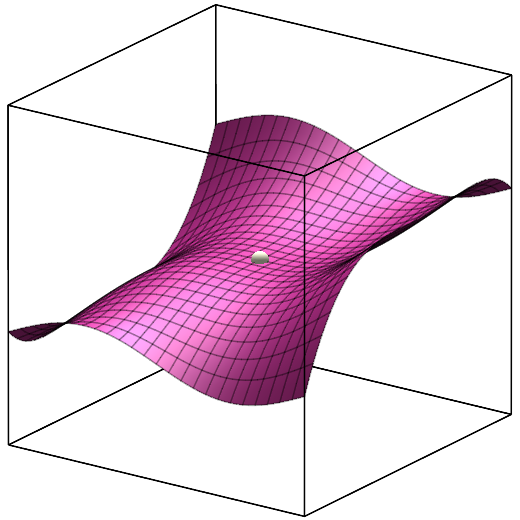}}}
        \put(7.1,3.1){\makebox(0.1,0.1)[bl]{$R_2$}}
        \put(7,0){\makebox(0,0)[bl]{\includegraphics[width=3\unitlength]{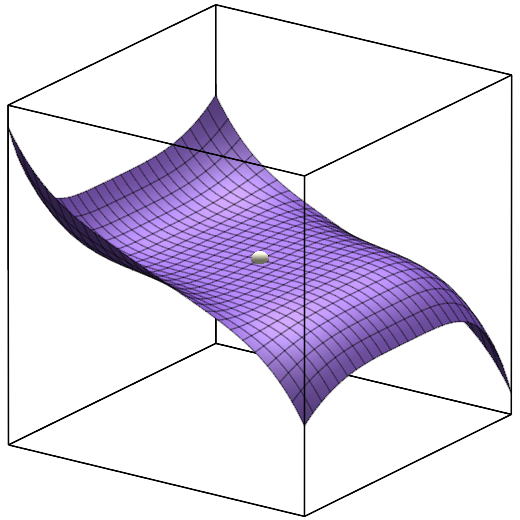}}}
        \put(7.1,0.1){\makebox(0,0)[bl]{$R_3$}}
      \end{picture}
	   \caption{Graph of the function $f(x,y)=12xy^4+12x^2+6xy+4\sin(\pi x)+3\sin(\pi y)-15$ over $[-1,1]^2$ (left), its quadratic and cubic Taylor polynomials about the midpoint $(0,0)$ (centre), and the corresponding remainders (right). The vertical range of all plot boxes is $[-30,30]$.}
	   \label{fig:Taylor-interpolation}
	\end{figure}
    
	Now let $B_0\in\iRR^2$ and $w_0=w(B_0)$. For any $B\subseteq B_0$ and
	$m,n\in\NN$ with $m\le n$, let us consider the Taylor expansion of
	$f$ about the midpoint $\bfm$ of $B$ and split off the first $m$
	terms from the Taylor polynomial $T_{n-1}$. For any $\bfx=(x,y)\in
	B$, we then have
	\[
	  f(\bfx) = T_{m-1}(\bfx)
	          + \sum_{k=m}^{n-1} \frac{1}{k!} d^k f(\bfm)
	          + \frac{1}{n!} d^n f(\bfxi_{\bfx})
	\]
	for some $\bfxi_{\bfx}\in B$ that depends on $\bfx$.
	Assume for simplicity that $B$ is a \emph{square} with $r\as
	r_x=r_y$. Since
	\[
	  h_x = x-m_x \in r [-1,1], \qquad
	  h_y = y-m_y \in r [-1,1],
	\]
	we can estimate the exact range of $f$ over $B$ by the \bemph{Taylor
	form of order $m$ and level $n$},
	\begin{equation}\label{eq:Taylor-form}
	  \intboxTaylor_{m,n} f(B) \as T_{m-1}(B) + r^m [-1,1] S_{m,n},\qquad
	  S_{m,n} \as \sum_{k=m}^n s_k r^{k-m},
	\end{equation}
	where the $s_k$ are defined as
	\[
	  s_k \as \frac{1}{k!} \sum_{j=0}^k \binom{k}{j}
	  	\abs{f^{(k-j,j)}(\bfm)}, \quad k=m,\dots,n-1, \qquad
	  s_n \as \frac{1}{n!} \sum_{j=0}^n \binom{n}{j}
	  	\abs{\intbox f^{(n-j,j)}(B)}
	\]
	and $\intbox f^{(n-j,j)}$, $j=0,\dots,n$ are bounded range functions
	over $B_0$.
	
	\begin{theorem}\label{theorem:Taylor-form}
	The Taylor form $\intboxTaylor_{m,n} f$ in~\eqref{eq:Taylor-form} has
	order $m$ convergence.
	\end{theorem}
	\begin{proof}
	Since $B\subseteq B_0$, we have $r=w(B)/2\le w_0/2$ and
	\[
	  s_k \le \frac{1}{k!} \sum_{j=0}^k \binom{k}{j}
	  	\abs{f^{(k-j,j)}(B_0)} \sa \bar{s}_k, \quad k=m,\dots,n-1.
	\]
	Moreover, since the range functions $\intbox f^{(n-j,j)}$ are
	bounded over $B_0$ there exist constants $C_j$, $j=0,\dots,n$, such
	that
	\[
	  s_n \le \frac{1}{n!} \sum_{j=0}^n \binom{n}{j} C_j \sa \bar{s}_n.
	\]
	Using these bounds, we conclude
	\[
	  S_{m,n} \le C', \qquad
	  C' \as \sum_{k=0}^{n-m} \Bigl( \frac{w_0}{2} \Bigr)^k \bar{s}_{k+m},
	\]
	where the constant $C'$ depends only on $f$ and $B_0$.
	
	Noticing that $\intboxTaylor_{m,n} f$ is a special case of the generalized
	Cornelius--Lohner form~\eqref{eq:CL-form-generalized} with $\intbox g=g=T_{m-1}$ and
	$\intbox R_g=r^m [-1,1] S_{m,n}$, the order $m$ convergence of
	$\intboxTaylor_{m,n} f$ then follows from
	Theorem~\ref{theorem:CL-form-generalized} and~\eqref{eq:strong-range-fn}, because
	\[
	  w(\intbox R_g(B))
	  \le 2 \abs{\intbox R_g(B)}
	   =  2 r^m S_{m,n}
	  \le C {w(B)}^m
	\]
    with $C\as2^{1-m}C'$.
	\end{proof}
	
	Note that $\intboxTaylor_{1,n}f$ is called ``Taylor form of
	\emph{order} $n$'' by \citet[Definition~3.4]{Ratschek:1984:CMF}, but
	``$n$-th order'' does not imply ``order $n$ convergence'' in their
	terminology~\citep[Section~2.4]{Ratschek:1984:CMF}. Instead, we
	prefer referring to this parameter as the \emph{level} of the Taylor
	form. The quadratic convergence of these Taylor forms (including the
	mean-value form $\intboxTaylor_{1,1}f$), which was proven by
	\citet[Theorem~3.4]{Ratschek:1984:CMF}, is also asserted by
	Theorem~\ref{theorem:Taylor-form}, because the exact range of the
	linear Taylor polynomial $T_1$ (see \ref{sec:range-linear-bivariate})
	is
	\[
	  T_1(B)
	  = f(\bfm) + r [-1,1] \bigl( \abs{f^{(1,0)}(\bfm)}
	  	+ \abs{f^{(0,1)}(\bfm)} \bigr)
	  = T_0(B) + r [-1,1] s_1,
	\]
	hence $\intboxTaylor_{1,n}f=\intboxTaylor_{2,n}f$.
	
	The advantage of our generalized Taylor forms is that they can
	achieve any convergence order in theory. In practice, however, one
	will probably make do with cubic or quartic convergence, because the
	exact ranges $T_2(B)$ and $T_3(B)$ of the quadratic and cubic Taylor
	polynomials can be determined by simple algorithms (see
	\ref{sec:range-quadratic-bivariate} and
	\ref{sec:range-cubic-bivariate}).
	
	If $f$ is a polynomial of degree $d$, so that all partial
	derivatives of order $k>d$ vanish, then we call $\intboxTaylor_m
	f(B)\as\intboxTaylor_{m,d+1}f(B)$ with $s_{d+1}=0$ the \bemph{maximal
	Taylor form of order $m$}. This form depends only on the radius $r$ of the square $B$ and
	the $(d+1)(d+2)/2\in O(d^2)$ partial derivatives of $f$ of orders
	$0,1,\dots,d$, evaluated at $\bfm$.

\section{Recursive Lagrange form}\label{sec:Lagrange-forms}

    \begin{figure}\centering\small
      \setlength{\unitlength}{0.1\linewidth}
      \begin{picture}(10,3)
        \put(0,0){\makebox(0,0)[bl]{\includegraphics[width=3\unitlength]{figs/f.png}}}
        \put(0.1,0.1){\makebox(0,0)[bl]{$f$}}
        \put(3.5,0){\makebox(0,0)[bl]{\includegraphics[width=3\unitlength]{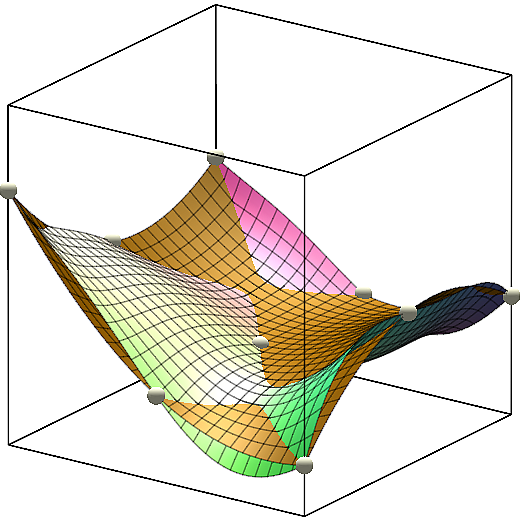}}}
        \put(3.6,0.1){\makebox(0,0)[bl]{$L_{0,0}$}}
        \put(7,0){\makebox(0,0)[bl]{\includegraphics[width=3\unitlength]{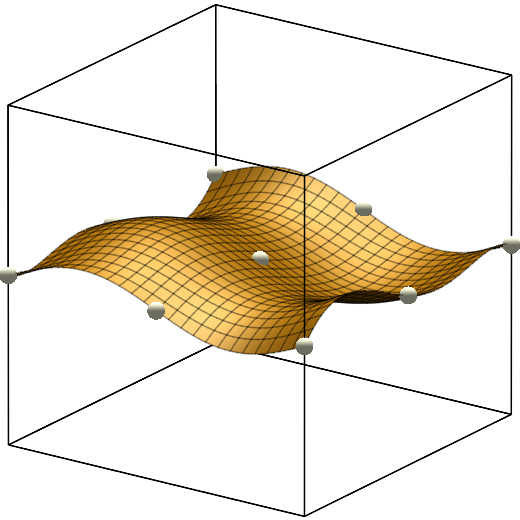}}}
        \put(6.8,0.1){\makebox(0,0)[bl]{$f-L_{0,0}$}}
      \end{picture}
	   \caption{Same function as in Figure~\ref{fig:Taylor-interpolation} (left), its biquadratic Lagrange interpolant (centre), and the remainder (right).}
	   \label{fig:Lagrange-interpolation}
	\end{figure}

    Another approach to construct range functions with cubic convergence
	is based on repeated Lagrange interpolation, an idea that was
	explored in the univariate setting by \citet{Hormann:2021:NRF}. To
	derive the bivariate version of their recursive Lagrange form, we
	recall a result by \citet[Theorem~2.1]{Moessner:2009:EBF}. It asserts
	that the error between $f$ and the biquadratic polynomial $L_{0,0}$
    (see Figure~\ref{fig:Lagrange-interpolation})
	that interpolates $f$ at the corners, the edge-midpoints, and the
	midpoint of a square $B$ with radius $\bfr=(r,r)$, that is, at the
	regular $3\times3$ grid
	$G(B)=\{m_x-r,m_x,m_x+r\}\times\{m_y-r,m_y,m_y+r\}$ (see
	\ref{sec:range-biquadratic}), is bounded from above by
	\[
	  \errorLagrange_{0,0}
	  \as \OmegaLagrange   \cdot \bigl(
	  		\abs{f^{(3,0)}(B)} + \abs{f^{(0,3)}(B)} \bigr)
	    + \OmegaLagrange^2 \cdot \abs{f^{(3,3)}(B)}, \qquad
	  \OmegaLagrange\as\frac{\sqrt{3}}{27}r^3,
	\]
	where the constant $\OmegaLagrange$ depends only on $r$. That is,
	\[
	  \abs{f(\bfx)-L_{0,0}(\bfx)} \le \errorLagrange_{0,0},
	\]
	for all $\bfx\in B$, hence
	\begin{equation}\label{eq:biquadratic-base}
	  f(B) \subseteq L_{0,0}(B) + [-1,1] \errorLagrange_{0,0}.
	\end{equation}
	
	To further bound $\errorLagrange_{0,0}$, we simply apply the same
	idea to the functions $f^{(3,0)}$ and $f^{(0,3)}$. That is, we
	determine the biquadratic polynomials $L_{1,0}$ and $L_{0,1}$ that
	interpolate $f^{(3,0)}$ and $f^{(0,3)}$, respectively, at $G(B)$, and
	conclude that
	\[
	  \abs{f^{(3,0)}(B)} \le \abs{L_{1,0}(B)}
	  	+ \errorLagrange_{1,0}, \qquad
	  \abs{f^{(0,3)}(B)} \le \abs{L_{0,1}(B)} + \errorLagrange_{0,1},
	\]
	where
	\begin{align*}
	  \errorLagrange_{1,0}
	  &\as \OmegaLagrange   \cdot \bigl( \abs{f^{(6,0)}(B)}
	  	+ \abs{f^{(3,3)}(B)} \bigr)
	     + \OmegaLagrange^2 \cdot \abs{f^{(6,3)}(B)},\\
	  \errorLagrange_{0,1}
	  &\as \OmegaLagrange   \cdot \bigl( \abs{f^{(3,3)}(B)}
	  	+ \abs{f^{(0,6)}(B)} \bigr)
	     + \OmegaLagrange^2 \cdot \abs{f^{(3,6)}(B)}.
	\end{align*}
	Consequently,
	\[
	  \begin{split}
	    \errorLagrange_{0,0}
	    &\le    \OmegaLagrange   \cdot \bigl( \abs{L_{1,0}(B)}
			+ \abs{L_{0,1}(B)} \bigr)\\
	    &\quad+ \OmegaLagrange^2 \cdot \bigl( \abs{f^{(6,0)}(B)}
			+ 3 \abs{f^{(3,3)}(B)} + \abs{f^{(0,6)}(B)} \bigr)
	        + \OmegaLagrange^3 \cdot \bigl( \abs{f^{(6,3)}(B)}
			+ \abs{f^{(3,6)}(B)} \bigr).
	  \end{split}
	\]
	
	If $f$ has bounded partial derivatives up to order $3n+3$, then we
	may repeat this procedure to obtain the upper bound
	\begin{equation}\label{eq:E00-bound}
	  \begin{split}
	    \errorLagrange_{0,0}
	    &\le    \sum_{k=1}^{n-1} \OmegaLagrange^k
			\sum_{j=0}^k \Delannoy{k}{j} \abs{L_{k-j,j}(B)}\\
	    &\quad+ \OmegaLagrange^n \sum_{j=0}^n
			\Delannoy{n}{j} \abs{f^{(3(n-j),3j)}(B)}
	          + \OmegaLagrange^{n+1} \sum_{j=1}^n
			  \Delannoy{n-1}{j-1} \abs{f^{(3(n+1-j),3j)}(B)},
	  \end{split}
	\end{equation}
	where $L_{i,j}$ denotes the biquadratic polynomial that interpolates
	$f^{(3i,3j)}$ at the grid $G(B)$. The coefficients
	\begin{equation}\label{eq:Delannoy}
	  \Delannoy{n}{k}
	  \as D(n-k,k)
	   =  \sum_{i=0}^k \binom{k}{i} \binom{n-k}{i} 2^i, \qquad
	  0 \le k \le n,
	\end{equation}
	are derived from the \emph{Delannoy numbers} $D(m,n)$. For $k=0$ and
	$k=n$, \eqref{eq:Delannoy} implies
	$\textDelannoy{n}{0}=\textDelannoy{n}{n}=1$. For other values of $k$,
	these coefficients satisfy a recurrence relation that is similar to
	the one for binomial coefficients,
	\[
	  \Delannoy{n}{k} = \Delannoy{n-1}{k-1} + \Delannoy{n-1}{k}
	  	+ \Delannoy{n-2}{k-1}, \qquad
	  0 < k < n.
	\]
	Arranging them in successive rows for $n=0,1,\dots$ gives the
	\emph{Tribonacci triangle} \citep{Alladi:1977:OTN}
	\[
	  \begin{array}{>{\centering\arraybackslash}m{15pt}*{12}
	  			{@{}>{\centering\arraybackslash}m{15pt}}}
	    &&&&&&           1                \\
	    &&&&&         1 & & 1             \\
	    &&&&       1 & & 3 & & 1          \\
	    &&&     1 & & 5 & & 5 & & 1       \\
	    &&   1 & & 7 & & 13 && 7 & & 1    \\
	    & 1 & & 9 && 25 & & 25 && 9 & & 1 \\[-1ex]
	    \text{$\iddots$} &&&&&& \text{$\vdots$} &&&&&& \text{$\ddots$}
	  \end{array}
	\]
	which is similar to \emph{Pascal's triangle}.
	
	Finally, we express each biquadratic Lagrange interpolant as
	$L_{i,j}=\TaylorLagrange_{i,j}+\remainderLagrange_{i,j}$, that is, as
	the sum of the second degree Taylor polynomial of $L_{i,j}$ about
	$\bfm$ (cf.~Equation~\eqref{eq:Taylor-polynomial}),
	\begin{equation}\label{eq:T_ij-Lagrange}
	  \begin{split}
	    \TaylorLagrange_{i,j}(\bfx)
	    = L_{i,j}(\bfm) &+ h_x L_{i,j}^{(1,0)}(\bfm)
			+ h_y L_{i,j}^{(0,1)}(\bfm)\\
	           &+ \frac12 h_x^2 L_{i,j}^{(2,0)}(\bfm)
				 	+ h_x h_y L_{i,j}^{(1,1)}(\bfm)
	                         + \frac12 h_y^2 L_{i,j}^{(0,2)}(\bfm),
	  \end{split}
	\end{equation}
	where $\bfx=(x,y)\in B$, $h_x=x-m_x$, and $h_y=y-m_y$, and the
	remainder
	\begin{equation}\label{eq:R_ij-Lagrange}
	  \remainderLagrange_{i,j}(\bfx)
	  = \frac12 h_x^2 h_y L_{i,j}^{(2,1)}(\bfm)
	  	+ \frac12 h_x h_y^2 L_{i,j}^{(1,2)}(\bfm)
	  	+ \frac14 h_x^2 h_y^2 L_{i,j}^{(2,2)}(\bfm).
	\end{equation}
	Note that simple algorithms for computing the exact ranges
	$\TaylorLagrange_{i,j}(B)$ and $\remainderLagrange_{i,j}(B)$ exist
	(see \ref{sec:range-quadratic-bivariate} and
	\ref{sec:range-biquadratic}). It then follows
	from~\eqref{eq:biquadratic-base} and~\eqref{eq:E00-bound} that $f(B)$
	can be estimated by the \bemph{recursive Lagrange form
    of order $3$ and level $n$},
	\begin{equation}\label{eq:Lagrange-form}
	  \intboxLagrange_{3,n} f(B)
	  \as \TaylorLagrange_{0,0}(B) + \remainderLagrange_{0,0}(B)
	  	+ [-1,1] U_{3,n}, \qquad
	  U_{3,n} \as \sum_{k=1}^{n+1} u_k \OmegaLagrange^k,
	\end{equation}
	where the $u_k$ are defined as
	\[
	  u_k
	  \as \sum_{j=0}^k \Delannoy{k}{j}
	  	\abs{\TaylorLagrange_{k-j,j}(B)
			+ \remainderLagrange_{k-j,j}(B)}, \quad
	  k=1,\dots,n-1,
	\]
	and
	\[
	  u_n \as \sum_{j=0}^n \Delannoy{n}{j}
	  	\abs{\intbox f^{(3(n-j),3j)}(B)},\qquad
	  u_{n+1} \as \sum_{j=1}^n \Delannoy{n-1}{j-1}
	  	\abs{\intbox f^{(3(n+1-j),3j)}(B)},
	\]
	and $\intbox f^{(3i,3j)}$ for $i+j\in\{n,n+1\}$ are bounded range
	functions over $B_0$.
	
	\begin{theorem}\label{theorem:Lagrange-form}
		The recursive Lagrange form $\intboxLagrange_{3,n} f$
		in~\eqref{eq:Lagrange-form} has cubic convergence.
	\end{theorem}
	\begin{proof}
	By Theorem~5.1 in \citep{Moessner:2009:EBF}, there exist constants
	$C_{k,l}$, such that the partial derivatives of the error between $f$
	and its biquadratic Lagrange interpolant $L_{0,0}$ are bounded from
	above as
	\[
	  \abs{ f^{(k,l)}(\bfx) - L_{0,0}^{(k,l)}(\bfx) }
	  \le C_{k,l} \Bigl[ {(2r)}^{3-k} \abs{f^{(3,l)}(B)}
	  		+ {(2r)}^{3-l} \abs{f^{(k,3)}(B)}
	                   + {(2r)}^{6-k-l} \abs{f^{(3,3)}(B)} \Bigr],
	\]
	for all $\bfx\in B$ and any $0\le k,l\le2$. Since $B\subseteq B_0$
	and $r=w(B)/2\le w_0/2$, this implies
	\[
	  \abs{L_{0,0}^{(k,l)}(\bfm)}
	  \le C'_{k,l},
	\]
	where the constant
	\[
	  C'_{k,l}
	  \as \abs{f^{(k,l)}(B_0)} + C_{k,l}
	  		\Bigl[ w_0^{3-k} \abs{f^{(3,l)}(B_0)}
	           + w_0^{3-l} \abs{f^{(k,3)}(B_0)}
	           + w_0^{6-k-l} \abs{f^{(3,3)}(B_0)} \Bigr].
	\]
	depends only on $f$ and $B_0$. By~\eqref{eq:R_ij-Lagrange}, we then
	have
	\[
	  \abs{\remainderLagrange_{0,0}(B)}
	  \le C' {w(B)}^3, \qquad
	  C' \as \frac1{16} ( C'_{2,1} + C'_{1,2} )
	  	+ \frac{1}{64} w(B) C'_{2,2}.
	\]
	
	Similarly, we can bound the partial derivatives of the error between
	$f^{(3i,3j)}$ and $L_{i,j}$ and conclude
	from~\eqref{eq:T_ij-Lagrange} and~\eqref{eq:R_ij-Lagrange} that
	$\abs{\TaylorLagrange_{i,j}(B)}$ and
	$\abs{\remainderLagrange_{i,j}(B)}$ are bounded from above by
	constants that depend only on $f$ and $B_0$. Since also the range
	functions $\intbox f^{(3i,3j)}$ are bounded over $B_0$, there exist
	constants $\bar{u}_k$ that depend only on $f$ and $B_0$, such that
	$u_k\le\bar{u}_k$ for $k=1,\dots,n+1$. Recalling that
	$\OmegaLagrange=\sqrt{3}/27\cdot r^3$, we conclude that
	\[
	  U_{3,n} \le C'' {w(B)}^3, \qquad
	  C'' \as \frac{\sqrt3}{216} \sum_{k=0}^n
	  	\biggl( \frac{\sqrt{3}}{216} w_0^3 \biggr)^k \bar{u}_{k+1}.
	\]
	
	Noticing that $\intboxLagrange_{3,n} f$ is a special case of the generalized
	Cornelius--Lohner form~\eqref{eq:CL-form-generalized} with
	$\intbox g=g=\TaylorLagrange_{0,0}$ and $\intbox
	R_g=\remainderLagrange_{0,0}(B) + [-1,1] U_{3,n}$, the cubic
	convergence of $\intboxLagrange_{3,n} f$ then follows from
	Theorem~\ref{theorem:CL-form-generalized} and~\eqref{eq:strong-range-fn}, because
	\[
	  w(\intbox R_g(B))
	  \le 2 \abs{\intbox R_g(B)}
	  \le 2 \abs{\remainderLagrange_{0,0}(B)} + 2 U_{3,n}
	  \le 2 (C'+C'') {w(B)}^3.
	\]
	\end{proof}
	
	If $f$ is a polynomial of degree $d$, then we call
	$\intboxLagrange_3 f(B)\as\intboxLagrange_{3,n+1}f(B)$ with
	$n=\floor{d/3}$ and $u_{n+1}=u_{n+2}=0$ the \bemph{maximal Lagrange
	form of order $3$}. This form depends only on the radius $r$ of the square $B$ and the
	$(n+1)(n+2)/2\in O(d^2)$ partial derivatives $f^{(3i,3j)}$ of $f$
	with $i,j\ge0$ and $i+j\le n$, evaluated at the 9 nodes of the grid $G(B)$, which is comparable to
	the $(d+1)(d+2)/2$ evaluations at the midpoint of $B$ that are needed
	for the maximal Taylor forms $\intboxTaylor_m f(B)$.
	
\section{Recursive Hermite form}\label{sec:Hermite-forms}

    \begin{figure}\centering\small
      \setlength{\unitlength}{0.1\linewidth}
      \begin{picture}(10,3)
        \put(0,0){\makebox(0,0)[bl]{\includegraphics[width=3\unitlength]{figs/f.png}}}
        \put(0.1,0.1){\makebox(0,0)[bl]{$f$}}
        \put(3.5,0){\makebox(0,0)[bl]{\includegraphics[width=3\unitlength]{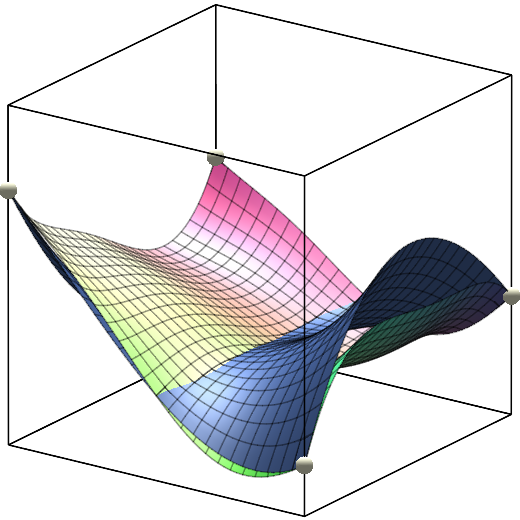}}}
        \put(3.6,0.1){\makebox(0,0)[bl]{$H_{0,0}$}}
        \put(7,0){\makebox(0,0)[bl]{\includegraphics[width=3\unitlength]{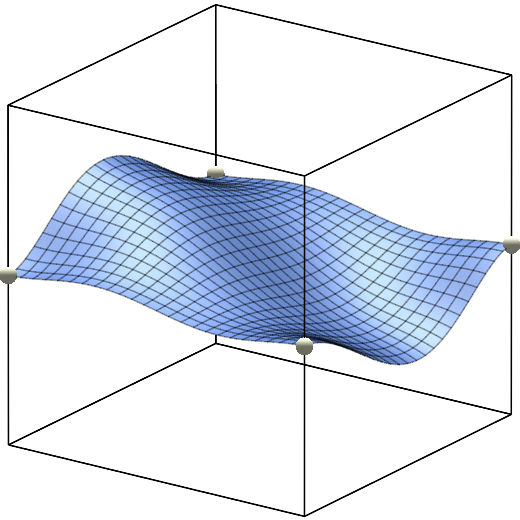}}}
        \put(6.8,0.1){\makebox(0,0)[bl]{$f-H_{0,0}$}}
      \end{picture}
	   \caption{Same function as in Figure~\ref{fig:Taylor-interpolation} (left), its bicubic Hermite interpolant (centre), and the remainder (right).}
	   \label{fig:Hermite-interpolation}
	\end{figure}

    A range function with quartic convergence can be obtained similarly
	by deriving the bivariate version of the univariate recursive Hermite
	form in \citep{Hormann:2023:RFO}. By Theorem~2.1 in
	\citep{Moessner:2009:EBF}, the error between $f$ and the bicubic
	polynomial $H_{0,0}$ (see Figure~\ref{fig:Hermite-interpolation})
    that interpolates $f$, $f^{(1,0)}$, $f^{(0,1)}$,
	and $f^{(1,1)}$ at the corners of the square $B$ with radius
	$\bfr=(r,r)$, is bounded from above by
	\[
	  \errorHermite_{0,0}
	  \as \OmegaHermite   \cdot \bigl( \abs{f^{(4,0)}(B)}
	  	+ \abs{f^{(0,4)}(B)} \bigr)
	    + \OmegaHermite^2 \cdot \abs{f^{(4,4)}(B)}, \qquad
	  \OmegaHermite\as\frac{1}{24}r^4.
	\]
	With the same arguments as in Section~\ref{sec:Lagrange-forms} and
	under the assumption that $f$ has bounded partial derivatives up to
	order $4n+4$, we can bound $\errorHermite_{0,0}$ from above by
	\begin{equation}\label{eq:E00-bound-Hermite}
	  \begin{split}
	    \errorHermite_{0,0}
	    &\le    \sum_{k=1}^{n-1} \Omega^k \sum_{j=0}^k
				\Delannoy{k}{j} \abs{H_{k-j,j}(B)}\\
	    &\quad+ \Omega^n \sum_{j=0}^n \Delannoy{n}{j}
				\abs{f^{(4(n-j),4j)}(B)}
	          + \Omega^{n+1} \sum_{j=1}^n \Delannoy{n-1}{j-1}
			  	\abs{f^{(4(n+1-j),4j)}(B)},
	  \end{split}
	\end{equation}
	where $H_{i,j}$ denotes the bicubic polynomial that interpolates
	$f^{(4i,4j)}$, $f^{(4i+1,4j)}$, $f^{(4i,4j+1)}$, and
	$f^{(4i+1,4j+1)}$ at the corners of $B$.
	
	As in Section~\ref{sec:Lagrange-forms}, it remains to split each
	bicubic Hermite interpolant into the third degree Taylor polynomial
	of $H_{i,j}$ about $\bfm$
	and the remainder,
	$H_{i,j}=\TaylorHermite_{i,j}+\remainderHermite_{i,j}$,
	noting that the exact ranges $\TaylorHermite_{i,j}(B)$ and
	$\remainderHermite_{i,j}(B)$ can be computed with simple procedures
	(see \ref{sec:range-cubic-bivariate} and \ref{sec:range-bicubic}).
	We can then estimate $f(B)$ by the \bemph{recursive Hermite form of order $4$ and level $n$},
	\begin{equation}\label{eq:Hermite-form}
	  \intboxHermite_{4,n} f(B)
	  \as \TaylorHermite_{0,0}(B) + \remainderHermite_{0,0}(B)
	  		+ [-1,1] V_{4,n}, \qquad
	  	V_{4,n} \as \sum_{k=1}^{n+1} v_k \OmegaHermite^k,
	\end{equation}
	where the $v_k$ are defined as
	\[
	  v_k \as \sum_{j=0}^k \Delannoy{k}{j}
	  		\abs{\TaylorHermite_{k-j,j}(B)
			+ \remainderHermite_{k-j,j}(B)}, \quad k=1,\dots,n-1,
	\]
	and
	\[
	  v_n \as \sum_{j=0}^n \Delannoy{n}{j}
	  	\abs{\intbox f^{(4(n-j),4j)}(B)},\qquad
	  v_{n+1} \as \sum_{j=1}^n \Delannoy{n-1}{j-1}
	  	\abs{\intbox f^{(4(n+1-j),4j)}(B)},
	\]
	and $\intbox f^{(4i,4j)}$ for $i+j\in\{n,n+1\}$ are bounded range
	functions over $B_0$.
	
	\begin{theorem}\label{theorem:Hermite-form}
	The recursive Hermite form $\intboxHermite_{4,n} f$
	in~\eqref{eq:Hermite-form} has quartic convergence.
	\end{theorem}
	\begin{proof}
	By Theorem~5.1 in \citep{Moessner:2009:EBF}, there exist constants
	$C_{k,l}$, such that the partial derivatives of the error between $f$
	and its bicubic Hermite interpolant $H_{0,0}$ are bounded from above
	as
	\[
	  \abs{ f^{(k,l)}(\bfx) - H_{0,0}^{(k,l)}(\bfx) }
	  \le C_{k,l} \Bigl[ {(2r)}^{4-k} \abs{f^{(4,l)}(B)}
	  		+ {(2r)}^{4-l} \abs{f^{(k,4)}(B)}
	                   + {(2r)}^{8-k-l} \abs{f^{(4,4)}(B)} \Bigr],
	\]
	for any $0\le k,l\le3$,
	hence
	\[
	  \abs{H_{0,0}^{(k,l)}(\bfm)}
	  \le C'_{k,l},
	\]
	where the constants $C'_{k,l}$ depend only on $f$ and $B_0$.
	
	
	Analogously to the proof of Theorem~\ref{theorem:Lagrange-form}, we
	conclude that there exist constants $C'$ and $C''$, such that
	\[
	  \abs{\remainderHermite_{0,0}(B)} \le C' {w(B)}^4, \qquad
	  V_{4,n} \le C'' {w(B)}^4,
	\]
	and by noticing that $\intboxHermite_{4,n} f$ is a special case of
	the generalized Cornelius--Lohner form~\eqref{eq:CL-form-generalized} with
	$\intbox g=g=\TaylorHermite_{0,0}$ and $\intbox R_g=\remainderHermite_{0,0}(B)
	+ [-1,1] V_{4,n}$, the quartic convergence of $\intboxHermite_{4,n}
	f$ then follows from Theorem~\ref{theorem:CL-form-generalized} and~\eqref{eq:strong-range-fn}, because
	\[
	  w(\intbox R_g(B))
	  \le 2 (C'+C'') {w(B)}^4.
	\]
	\end{proof}

	If $f$ is a polynomial of degree $d$, then we call
	$\intboxHermite_4 f(B)\as\intboxHermite_{4,n+1}f(B)$ with
	$n=\floor{d/4}$ and $v_{n+1}=v_{n+2}=0$ the \bemph{maximal Hermite
	form of order $4$}. This form depends only on the radius $r$ of the square $B$ and the
	$(n+1)(n+2)/2\in O(d^2)$ partial derivatives $f^{(4i+k,4j+l)}$ of $f$
	with $i,j\ge0$ and $i+j\le n$ and $0\le k,l\le1$, evaluated at the 4 corners of $B$, which is comparable to number of evaluations needed
	for the maximal Taylor forms $\intboxTaylor_m f(B)$ 
    and the maximal Lagrange forms $\intboxLagrange_m f(B)$.

\section{Experiments}\label{sec:experiments}

\begin{table}[t!]\small\centering
	  \begin{tabular}{lcccl}
	    \toprule
	    name & zero level set & graph & degree & $f(x,y)$\\
	    \midrule
	    clover-4 & \parbox{0.8in}{\includegraphics{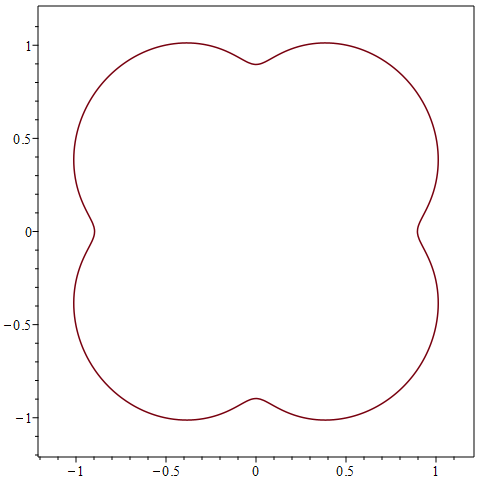}}
	             & \parbox{0.95in}{\includegraphics{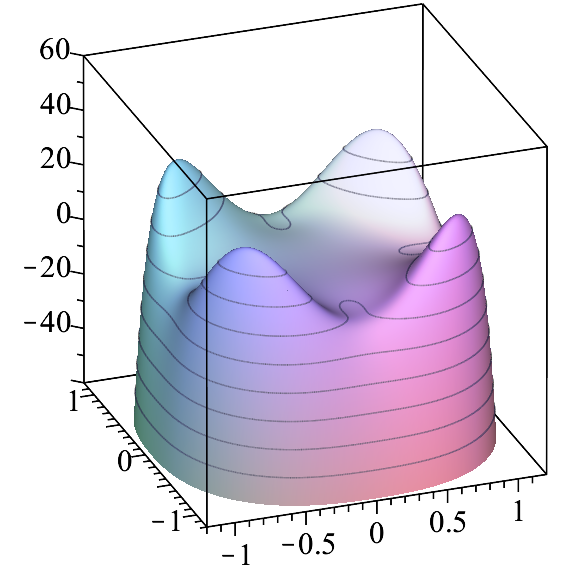}}
	             & $10$ & \scalebox{1}{\parbox{217pt}{
	        $- 50x^{10} - (249y^2 - 57) x^8 - (498y^4 - 227y^2 - 1)x^6$\\
	        $- (498y^6 - 341y^4 - 3y^2 + 16)x^4$\\
	        $- (249y^8 - 227y^6 - 3y^4 - 102y^2 - 1)x^2$\\
	        $- 50y^{10} + 57y^8 + y^6 - 16y^4 + y^2 + 1$}}\\
	    clover-5 & \parbox{0.8in}{\includegraphics{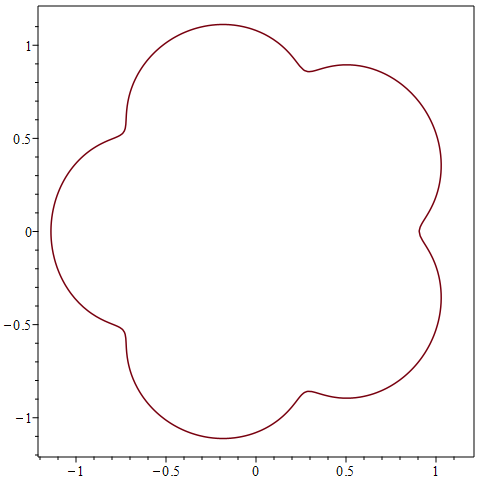}}
	             & \parbox{0.95in}{\includegraphics{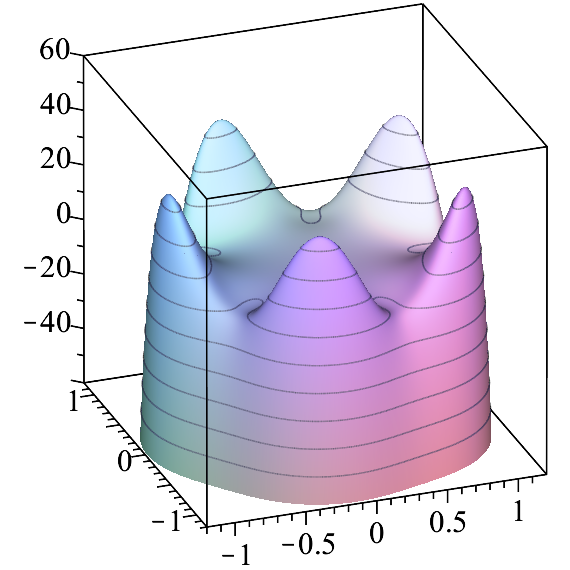}}
	             & $12$ & \scalebox{0.9}{\parbox{240pt}{
	        $- 71x^{12} - (424y^2 - 79)x^{10} - (1059y^4
					- 396y^2 - 1)x^8$\\
	        $- (1412y^6 - 793y^4 - 4y^2 - 1)x^6 - 20x^5$\\
	        $- (1059y^8 - 793y^6 - 6y^4 - 3y^2 - 1)x^4 + 202y^2x^3$\\
	        $- (424y^{10} - 396y^8 - 4y^6 - 3y^4 - 2y^2 - 1)x^2$\\
	        $- 101y^4x - 71y^{12} + 79y^{10} + y^8 + y^6 + y^4
					+ y^2 + 1$}}\\
	    clover-8 & \parbox{0.8in}{\includegraphics{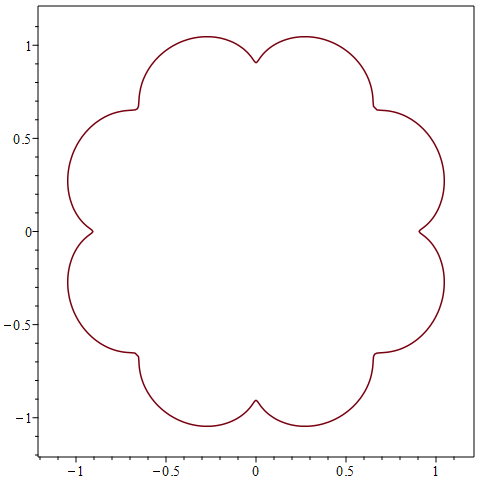}}
	             & \parbox{0.95in}{\includegraphics{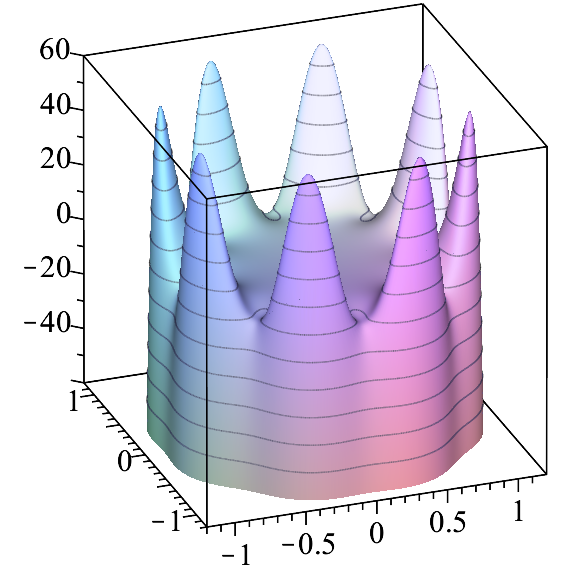}}
	             & $18$ & \scalebox{0.7}{\parbox{310pt}{
	        $- 156x^{18} - (1406y^2 - 170)x^{16} - (5625y^4
					- 1363y^2 - 1)x^{14}$\\
	        $- (13125y^6 - 4769y^4 - 7y^2 - 1)x^{12}
					- (19688y^8 - 9538y^6 - 21y^4 - 6y^2 - 1)x^{10}$\\
	        $- (19688y^{10} - 11922y^8 - 35y^6 - 15y^4 - 5y^2 + 30)x^8$\\
	        $- (13125y^{12} - 9538y^{10} - 35y^8
					- 21y^6 - 11y^4 - 879y^2 - 1)x^6$\\
	        $- (5625y^{14} - 4769y^{12} - 21y^{10}
					- 15y^8 - 11y^6 + 2181y^4 - 4y^2 - 1)x^4$\\
	        $- (1406y^{16} - 1363y^{14} - 7y^{12} - 6y^{10}
					- 5y^8 - 879y^6 - 4y^4 - 3y^2 - 1)x^2$\\
	        $- 156y^{18} + 170y^{16} + y^{14} + y^{12} + y^{10}
					- 30y^8 + y^6 + y^4 + y^2 + 1$}}\\
	    grass & \parbox{0.8in}{\includegraphics{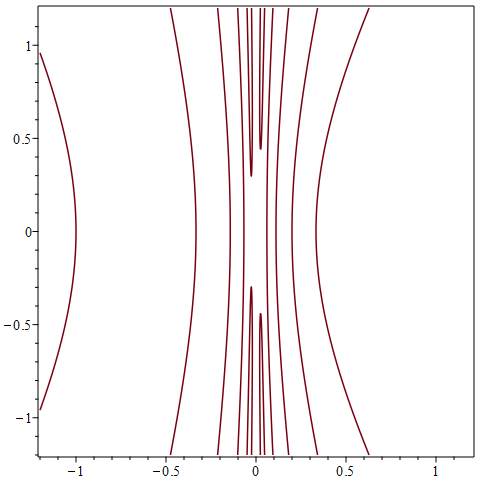}}
	          & \parbox{0.95in}{\includegraphics{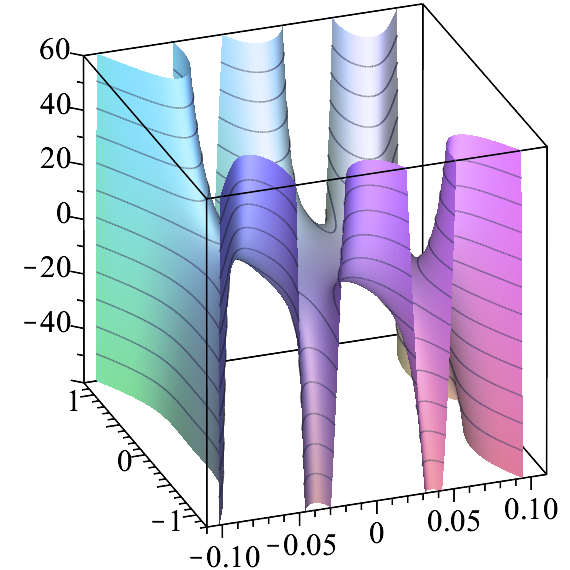}}
	          & $12$ & \scalebox{1}{\parbox{217pt}{
	        $\displaystyle 1 + \prod_{k=1}^6
				\bigl( (1-4^k)x^2 + y^2 - 2x + 1 \bigr)$}}\\
	    \bottomrule
	  \end{tabular}
		\caption{Overview of the four polynomial test functions used in our
		numerical experiments.}
\label{tab:test-functions}
\end{table}

    In this section, we validate the practicality and
	usefulness of our new range functions.
	We do this by reporting on the raw performance
	of evaluating each range function on individual boxes.
	We evaluate these range functions in terms of
	efficiency and efficacy as noted in the introduction.
	Our code, data, and Makefile experiments can 
    be downloaded from our Core Library
    webpage~\citep{core:home-range}.

	Our experimental platform is a Windows 11 laptop equipped with a 2.2
	GHz Intel Core i7-14650HX processor and 32 GB of RAM.  All numerical
	experiments are implemented in Julia using standard IEEE 754
	double-precision floating-point arithmetic (Float64), and
	timings are obtained using Julia’s built-in benchmarking tools.
	Of independent interest is that 
	our bivariate functions are implemented as
	Straight-Line Programs (SLPs) or codelists
	in the sense of Automatic Differentiation
	\citep{griewank-walther:bk},
	and we do symbolic differentiation of codelists
	(in contrast to the standard numeric differentiation)
	to improve performance.
	We will report about the details of this implementation in a separate paper.

\xCY{Ignore this?
	{\em Although all higher derivatives of these functions can 
	be automatically evaluated, we actually take the unusual
	step of symbolically generating these derivatives (generating
	new SLPs).  The reason is that, in the context of
	our subdivision applications, this allows us to reuse these
	derivatives just by evaluation over a multitude of boxes.}
	\cored{Bingwei/Thomas:} do we have experiments showing
	the superiority of this approach?  APPARENTLY, the
	answer is NO.  But it should be fixed!
	}

\subsection{Test functions}
	Table~\ref{tab:test-functions}
    lists \rev{four bivariate polynomials for which the results
    reported below capture the typical behaviour observed
    in all our experiments.}
	For each test function $f(x,y)$ we show 
    the curve $f(x,y)=0$ in the domain $B_0\as[-1.2,1.2]^2$, as
    well as the graph of $f$, to
	give some intuition about the characteristics of
    these polynomials. 
    
\subsection{Convergence orders}

\begin{figure}[t!]\footnotesize
	  \begin{tabular}{rcccc}
	    \toprule
	    & \multicolumn{2}{c}{$r=0.1$} & \multicolumn{2}{c}{$r=0.01$}
				\\\cmidrule(lr){2-3}\cmidrule(lr){4-5}
	    & range & $\Hausdorff$ & range & $\Hausdorff$ \\
	    \midrule
	            $f$ & $[-1.3586, -0.9646]$ &   --- 
					& $[-1.07788547, -1.05241970]$ & --- \\
	    \cmidrule{1-5}
	    $\intboxTaylor_2   f$ & $[-1.4303, -0.6978]$ & $0.2667$
				& $[-1.07824745, -1.04988220]$ & $0.00253750$\\[\jot]
	    $\intboxTaylor_3   f$ & $[-1.3976, -0.8436]$ & $0.1209$
				& $[-1.07792045, -1.05238265]$ & $0.00003705$\\[\jot]
	    $\intboxLagrange_3 f$ & $[-1.3688, -0.8688]$ & $0.0958$
				& $[-1.07789250, -1.05241267]$ & $0.00000703$\\[\jot]
	    $\intboxTaylor_4   f$ & $[-1.3630, -0.9397]$ & $0.0249$
				& $[-1.07788591, -1.05241719]$ & $0.00000252$\\[\jot]
	    $\intboxHermite_4  f$ & $[-1.3621, -0.9508]$ & $0.0138$
				& $[-1.07788571, -1.05241821]$ & $0.00000149$\\
	    \bottomrule
	  \end{tabular}\hfill
	  \parbox{0.3\linewidth}{\includegraphics{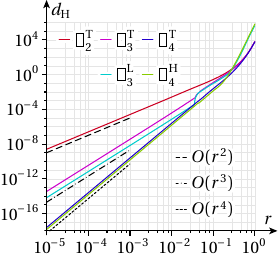}}
	  \caption{Concrete (rounded) values (left) and log-log plot (right)
	  of the Hausdorff distance $\Hausdorff$ between the exact range
	  $f(B_r)$ of the test function ``clover-4''
	  over a square $B_r$ with midpoint
	  $(0.1,0.2)$ and radius $r\in[0.00001,1]$ and the approximate ranges
	  $\intbox^\mathrm{X}_m f(B_r)$ given by maximal Taylor, Lagrange, and Hermite
	  forms.}
\label{fig:convergence}
\end{figure}

\begin{figure}[t!]\footnotesize\def\z{\phantom{0}}
	  \begin{tabular}{rcccc}
	    \toprule
	    & \multicolumn{2}{c}{$r=0.005$} & \multicolumn{2}{c}{$r=0.0005$}
				\\\cmidrule(lr){2-3}\cmidrule(lr){4-5}
	    & range & $\Hausdorff$ & range & $\Hausdorff$ \\
	    \midrule
	            $f$ & $[-61.874, -46.411]$ &   --- 
					& $[-60.5351611, -59.2710915]$ & --- \\
	    \cmidrule{1-5}
	    $\intboxTaylor_2   f$ & $[-73.566, -46.367]$ & $11.6914$
				& $[-60.6614110, -59.2708307]$ & $0.12624978$\\[\jot]
	    $\intboxTaylor_3   f$ & $[-62.737, -46.391]$ & \z$0.8625$
				& $[-60.5351831, -59.2710780]$ & $0.00002195$\\[\jot]
	    $\intboxLagrange_3 f$ & $[-62.639, -45.980]$ & \z$0.7648$
				& $[-60.5355311, -59.2707216]$ & $0.00036989$\\[\jot]
	    $\intboxTaylor_4   f$ & $[-61.926, -46.404]$ & \z$0.0516$
				& $[-60.5351702, -59.2710910]$ & $0.00000904$\\[\jot]
	    $\intboxHermite_4  f$ & $[-61.947, -46.360]$ & \z$0.0728$
				& $[-60.5351657, -59.2710865]$ & $0.00000503$\\
	    \bottomrule
	  \end{tabular}\hfill
	  \parbox{0.3\linewidth}{\includegraphics{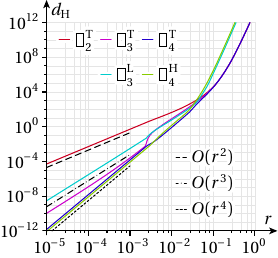}}
	  \caption{Data and plot for a similar experiment as in Figure~\ref{fig:convergence}, using the ``grass'' 
      test function and a square with midpoint $(0.1,0.1)$.}
\label{fig:convergence-grass}
\end{figure}

	Figures~\ref{fig:convergence} and~\ref{fig:convergence-grass} report
	and compare the tightness of the five range functions
		\begin{equation}\label{eq:5ranges}
			\intboxTaylor_2 f,\quad \intboxTaylor_3 f,\quad \intboxTaylor_4 f,
				\quad \intboxLagrange_3 f, \quad \intboxHermite_4 f,
		\end{equation}
	where $f$ is the function ``clover-4'' (Figure~\ref{fig:convergence}) or ``grass'' (Figure~\ref{fig:convergence-grass}) from Table~\ref{tab:test-functions}
	and $\intbox^{\mathrm{X}}_m f$ (for $\mathrm{X}\in \{\mathrm{T},\mathrm{L},\mathrm{H}\}$)
	is based on the maximal $\mathrm{X}$-form of order $m$.
	We evaluate it over squares of varying radii around a
	specific midpoint. The plots confirm the asymptotic convergence
	orders of the Taylor forms of orders $m=2,3,4$, the Lagrange form of
	order $3$, and the Hermite form of order $4$, as stated in
	Theorems~\ref{theorem:Taylor-form}, \ref{theorem:Lagrange-form}, and
	\ref{theorem:Hermite-form}, respectively. We further observe in Figure~\ref{fig:convergence}, both
	from the plot and the data, that $\intboxLagrange_3$ is consistently
	tighter than $\intboxTaylor_3$ for $r\le0.275$ and significantly
	tighter for $r\le0.035$. Similarly, $\intboxHermite_4$ is tighter
	than $\intboxTaylor_4$ for $r\le0.02$. 
    While the convergence orders were confirmed in all our
	tests, a different relative behaviour between $\intboxLagrange_3$ and $\intboxTaylor_3$ can be seen in Figure~\ref{fig:convergence-grass}.
    \rev{This reflects the fact that the Hausdorff distance $\Hausdorff$ between the exact and the approximate range depends not only on the convergence \emph{range}, but also on the convergence \emph{constant}, which in turn depends on the test function $f$.
    Here, we can deduce from the plots that the 
    convergence constant of
    $\intboxTaylor_3$ for ``clover-4'' is bigger than 
    that of $\intboxLagrange_3$ for ``clover-4'', but
    that this relation is inverted for ``grass''.}

\subsection{Efficiency and efficacy}
	In Table~\ref{tab:LT}, we evaluate the five range functions
	of \eqref{eq:5ranges} on the set of
	$1024$ boxes obtained by subdividing
	the square $B_0=[-1.2,1.2]^2$ uniformly into a $32\times 32$ grid.
	Let $\totalTime^\mathrm{X}_m$ be the sum of the times for evaluating
	$\intbox^\mathrm{X}_m f(B)$ each of these $1024$ boxes.
	Similarly, $\totalWidth^\mathrm{X}_m$ is the sum of the widths
	of the ranges for these $1024$ boxes.
	Viewing the quadratically convergent $\intboxTaylor_2$ as
	the \emph{base line} method, we define the \emph{speedup ratio}
	(i.e., efficiency ratio)
	of method $\intbox^\mathrm{X}_m$ as
		$\frac{\totalTime^\mathrm{T}_2}{\totalTime^\mathrm{X}_m}$.
	Similarly, the \emph{efficacy ratio} 
	of method $\intbox^\mathrm{X}_m$ is
		$\frac{\totalWidth^\mathrm{T}_2}{\totalWidth^\mathrm{X}_m}$.
	Note that we define each ratio so that the method $\intbox^\mathrm{X}_m$
	is better than the base line if and only if the ratio is greater than $1$.
	The column ``time (ms)'' gives the average time in milliseconds over 10 runs for 
    each method and test function, as measured by Julia's benchmarking tool.

	One of the attractions of Lagrange and Hermite ranges
	is that in subdivision algorithms,
	their evaluations can \emph{share} computations
	\citep{Hormann:2021:NRF,Hormann:2023:RFO}, in this
    case the evaluation of the test function and its partial 
    derivatives at the corners and edge midpoints of the 1024 boxes,
	leading to amortized time and space complexity
	that is not reflected in the timing of individual evaluations.
	So we added a second row in Table~\ref{tab:LT} for these
	two methods to show
	the improved speedup and memory usage when sharing is turned on.

\def\z{\phantom{0}}
\begin{table}[t!]
	\centering\small
	\begin{tabular}{llcccc}
	\toprule
	test function & range function & time (ms) & speedup & efficacy & memory (MB) \\
	\midrule
    \multirow{7}{*}{clover-4}
    & $\intboxTaylor_2$ (baseline) & \z195.98 & 1 & 1 & \z141.70 \\\cmidrule(lr){2-6}
    & $\intboxTaylor_3$ & \z197.92 & 0.99 & 1.1978 & \z140.89 \\
    & $\intboxTaylor_4$ & \z233.58 & 0.84 & 1.1991 & \z158.68 \\\cmidrule(lr){2-6}
    & $\intboxLagrange_3$
    & \z368.19 & 0.53 & \multirow{2}{*}{1.1950} & \z251.10 \\
    & $\intboxLagrange_3$ (shared)
    & \color{mygreen} \z179.67 & \color{mygreen} 1.09 & & \color{mygreen} \z122.66 \\\cmidrule(lr){2-6}
    & $\intboxHermite_4$
    & \z546.54 & 0.36 & \multirow{2}{*}{\color{mygreen} 1.1997} & \z748.78 \\
    & $\intboxHermite_4$ (shared)
    & \z306.17 & 0.64 & & \z584.38 \\
    \midrule
    \multirow{7}{*}{clover-5}
    & $\intboxTaylor_2$ (baseline) & \z333.99 & 1 & 1 & \z217.26 \\\cmidrule(lr){2-6}
    & $\intboxTaylor_3$ & \z312.35 & 1.07 & 1.2223 & \z216.45 \\
    & $\intboxTaylor_4$ & \z320.70 & 1.04 & 1.2229 & \z235.58 \\\cmidrule(lr){2-6}
    & $\intboxLagrange_3$
    & \z592.08 & 0.56 & \multirow{2}{*}{1.2195} & \z398.05 \\
    & $\intboxLagrange_3$ (shared)
    & \color{mygreen} \z288.54 & \color{mygreen} 1.16 & & \color{mygreen} \z193.40 \\\cmidrule(lr){2-6}
    & $\intboxHermite_4$
    & \z865.50 & 0.39 & \multirow{2}{*}{\color{mygreen} 1.2240} & 1039.23 \\
    & $\intboxHermite_4$ (shared)
    & \z457.68 & 0.73 & & \z801.42 \\
    \midrule	
    \multirow{7}{*}{clover-8}
    & $\intboxTaylor_2$ (baseline) & \z846.72 & 1 & 1 & \z560.39 \\\cmidrule(lr){2-6}
    & $\intboxTaylor_3$ & \z857.81 & 0.99 & 1.2986 & \z559.57 \\
    & $\intboxTaylor_4$ & \z848.14 & 1.00 & 1.2990 & \z584.61 \\\cmidrule(lr){2-6}
    & $\intboxLagrange_3$
    & 1514.08 & 0.56 & \multirow{2}{*}{1.2941} & \z975.60 \\
    & $\intboxLagrange_3$ (shared)
    & \color{mygreen} \z755.13 & \color{mygreen} 1.12 & & \color{mygreen} \z466.93 \\\cmidrule(lr){2-6}
    & $\intboxHermite_4$
    & 1928.57 & 0.44 & \multirow{2}{*}{\color{mygreen} 1.3014} & 2222.16 \\
    & $\intboxHermite_4$ (shared)
    & \z988.37 & 0.86 & & 1625.63 \\
    \midrule
    \multirow{7}{*}{grass}
    & $\intboxTaylor_2$ (baseline) & \z838.99 & 1 & 1 & \z492.87 \\\cmidrule(lr){2-6}
    & $\intboxTaylor_3$ & \z901.25 & 0.93 & 1.1993 & \z492.06 \\
    & $\intboxTaylor_4$ & \z804.90 & 1.04 & \color{mygreen} 1.2014 & \z511.20 \\\cmidrule(lr){2-6}
    & $\intboxLagrange_3$
    & 1385.02 & 0.61 & \multirow{2}{*}{1.1890} & \z807.83 \\
    & $\intboxLagrange_3$ (shared)
    & \color{mygreen} \z662.46 & \color{mygreen} 1.27 & & \color{mygreen} \z381.26 \\\cmidrule(lr){2-6}
    & $\intboxHermite_4$
    & 1608.16 & 0.52 & \multirow{2}{*}{1.2008} & 1469.37 \\
    & $\intboxHermite_4$ (shared)
    & \z669.50 & 1.25 & & \z916.84 \\
    \bottomrule
    \end{tabular}
	\caption{Performance comparison of the range functions in~\eqref{eq:5ranges} for
    the 1024 boxes obtained by uniformly subdividing the square $[-1.2,1.2]^2$.
	For each test function, we highlight in {\color{mygreen}green}
	the best values under \emph{time},
	\emph{speedup}, \emph{efficacy}, and \emph{memory} usage.}
\label{tab:LT}
\end{table}

	From the data in Table~\ref{tab:LT} we observe
    that the efficacy (tightness) is greater than $1$ for all
	the methods with superior convergence order, as expected.
	In fact, efficacy ratios lie in the range $[1.19,1.30]$.
    At least for $\intboxTaylor_3$ and $\intboxTaylor_4$, this
    improvement does not come at a significant loss in efficiency,
    since the speedups are only marginally less than $1$ in some 
    cases and often even greater than $1$ (i.e., they are a strict 
    improvement on the base line). Likewise, the memory usage is very similar
    for all three Taylor forms.
    %
    %
	Instead, for $\intboxLagrange_3$ and $\intboxHermite_4$ (without sharing), the
    tradeoff is clearly visible, both in the speedup ratios, which are in the 
    ranges $[0.52,0.61]$ and $[0.36,0.52]$, respectively, and the memory usage. 
    But speedup is roughly double (and memory roughly halved), if the data is shared.
    %
    
    We draw two general conclusions (within the limits of our test functions):
	(1) it appears that the efficacy grows with the convergence order $m$, but the improvement from $m=3$ to $m=4$ is not significant; 
    (2) the efficiency (both in terms of runtime and memory usage) 
    of the higher order Taylor forms 
    $\intboxTaylor_3$ and $\intboxTaylor_4$ is similar to 
    the base line $\intboxTaylor_2$, but the Lagrange form 
    $\intboxLagrange_3$ with data sharing is the most efficient method.

\subsection{Pairwise efficacy comparison over a landscape}


	While the previous table uses an aggregate $\totalWidth^\mathrm{X}_m$
	to measure the efficacy of $\intbox^\mathrm{X}_m$ against the base line,
	let us now compare the relative efficacy of
	two range functions
			$\intbox^\mathrm{X}_m f$ and $\intbox^\mathrm{Y}_n f$
	over the entire landscape of $1024$ boxes as follows.
	For each box $B$ in the $32\times 32$ grid of $B_0=[-1.2,1.2]^2$,
	we compute the value
		\[
		W^{\mathrm{X},\mathrm{Y}}_{m,n} (B)
        \as \log_{10} 
			\frac{w\bigl(\intbox^\mathrm{X}_m f(B)\bigr)}{w\bigl(\intbox^\mathrm{Y}_n f(B)\bigr)}.
		\]
	Thus, the first range $\intbox^\mathrm{X}_m f$ is tighter than the second $\intbox^\mathrm{Y}_n f$
	if and only if $W^{\mathrm{X},\mathrm{Y}}_{m,n}<0$.
	We can now visualize their relative tightness over the
	entire landscape of $1024$ boxes by mapping
	the values $W^{\mathrm{X},\mathrm{Y}}_{m,n}$ to a color scale that ranges
	from dark green (for the smallest value) to dark red (for the biggest value).
    In general, green hues indicate boxes for which $\intbox^\mathrm{X}_m f$ is tighter,
    while boxes for which $\intbox^\mathrm{Y}_n f$ is tighter are in red hues, and yellow hues are used for boxes where both ranges are similar.
	For example, a mostly green landscape as in Figure~\ref{fig:clover4568gT32}
	indicates that the first function is generally tighter
	than the second.
    
	We consider four such pairwise comparisons of methods (over
	all four test functions), which confirm the observations from
    Table~\ref{tab:LT}. Figure~\ref{fig:clover4568gT32} shows that
    $\intboxTaylor_3$ gives significantly tighter ranges than the
    baseline $\intboxTaylor_2$, essentially for all 1024 boxes. 
    The absence of red hues in Figure~\ref{fig:clover4568gT43} 
    further shows that $\intboxTaylor_4$ is generally tighter than
    $\intboxTaylor_3$, but the predominantly yellow landscape also
    indicates that the average improvement is not very substantial.  However, the green hues along the outlines of the clovers and grass suggest that $\intboxTaylor_4$ is tighter near
    the zero set of $f$ and would be more effective in tracing these curves.
    Figure~\ref{fig:clover4568gT3L3} shows that the two third-order
    methods $\intboxLagrange_3$ and $\intboxTaylor_3$ can both
    outperform each other in terms of tightness and neither can be
    declared a clear winner in general. The outcome seems to depend
    a lot on the local shape of the function over each box. As we 
    can see in Figure~\ref{fig:clover4568gT4H4}, a similar
    verdict holds for the two fourth-order methods $\intboxHermite_4$ 
    and~$\intboxTaylor_4$. 
    
\subsection{Experimental limitations}
	Although Theorem~\ref{theorem:CL-form-generalized}
	provides a complete CL theory that can be implemented
	in practical arithmetic models (say, bigFloat libraries),
    our implementation adopts the ``standard expedient'' of
	computing with the IEEE arithmetic model (Julia's Float64).
At present, our SLP implementation only supports
polynomial functions.
Nevertheless, the underlying theory developed in this paper
applies to general analytic functions.
	More precisely,
	to compute $\intbox f(B)=\intbox g(B)+\intbox R_g(B)$,
	we compute $\intbox g(B)$ (and similarly for $\intbox R_g(B)$)
	as follows:
	If $g(B)=[a,b]$ and $g(x,y)$ is a low degree polynomial,
	then we have exact algebraic expressions
	for $a$ and $b$ (see the Appendices).
	Then a valid range function would be
	$\intbox g(B)=[\tilde a -\varepsilon,\tilde b +\varepsilon]$
	where $\tilde a, \tilde b$ are IEEE evaluation of the expressions
	for $a,b$ and $\abs{a-\tilde{a}},\abs{b-\tilde{b}}\le\varepsilon$.  
    Although $\varepsilon$ could be
	determined for any $g$, we took the shortcut of
	ignoring this $\varepsilon$-correction.   
    Let us suppose that all these errors are less than some $\varepsilon_0$.  Is this amount of error still  ``correct''?  This question can only be answered in specific applications of these range functions.

\begin{figure}[t!]
    \includegraphics[width=1.28in,height=1.28in]{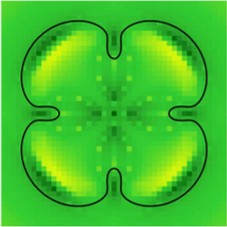}\hfill%
    \includegraphics[width=1.28in,height=1.28in]{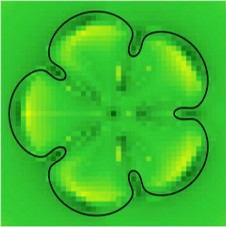}\hfill%
    \includegraphics[width=1.28in,height=1.28in]{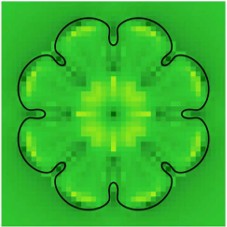}\hfill%
    \includegraphics[width=1.28in,height=1.28in]{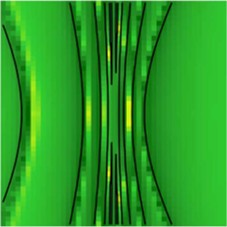}\hfill%
    \includegraphics[width=0.36in,height=1.28in]{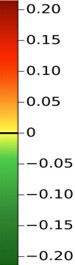}\\
    \parbox{1.28in}{\centering clover-4}\hfill%
    \parbox{1.28in}{\centering clover-5}\hfill%
    \parbox{1.28in}{\centering clover-8}\hfill%
    \parbox{1.28in}{\centering grass}\hfill\parbox{0.36in}{~}
    \caption{Efficacy comparison between $\intboxTaylor_3$ and $\intboxTaylor_2$.}
    \label{fig:clover4568gT32}
\end{figure}
	
	
\begin{figure}[t!]
    \includegraphics[width=1.28in,height=1.28in]{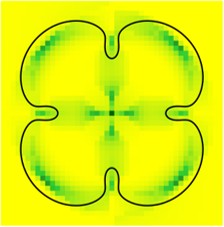}\hfill%
    \includegraphics[width=1.28in,height=1.28in]{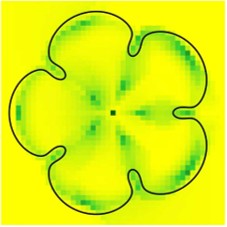}\hfill%
    \includegraphics[width=1.28in,height=1.28in]{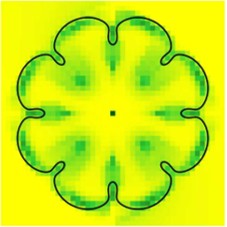}\hfill%
    \includegraphics[width=1.28in,height=1.28in]{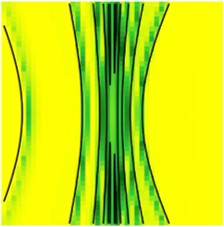}\hfill%
    \includegraphics[width=0.36in,height=1.28in]{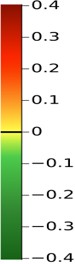}\\
    \parbox{1.28in}{\centering clover-4}\hfill%
    \parbox{1.28in}{\centering clover-5}\hfill%
    \parbox{1.28in}{\centering clover-8}\hfill%
    \parbox{1.28in}{\centering grass}\hfill\parbox{0.36in}{~}
    \caption{Efficacy comparison between $\intboxTaylor_4$ and $\intboxTaylor_3$.}
	\label{fig:clover4568gT43}
\end{figure}

	
\begin{figure}[t!]
    \includegraphics[width=1.28in,height=1.28in]{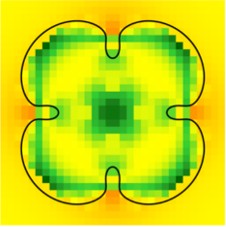}\hfill%
    \includegraphics[width=1.28in,height=1.28in]{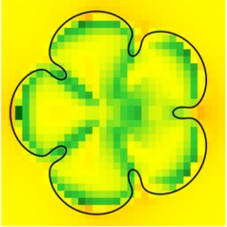}\hfill%
    \includegraphics[width=1.28in,height=1.28in]{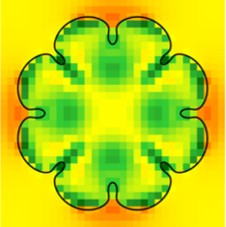}\hfill%
    \includegraphics[width=1.28in,height=1.28in]{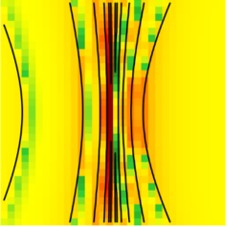}\hfill%
    \includegraphics[width=0.36in,height=1.28in]{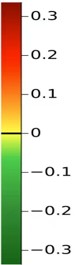}\\
    \parbox{1.28in}{\centering clover-4}\hfill%
    \parbox{1.28in}{\centering clover-5}\hfill%
    \parbox{1.28in}{\centering clover-8}\hfill%
    \parbox{1.28in}{\centering grass}\hfill\parbox{0.36in}{~}
    \caption{Efficacy comparison between $\intboxLagrange_3$ and $\intboxTaylor_3$.}
    \label{fig:clover4568gT3L3}
\end{figure}
	
	
\begin{figure}[t!]
    \includegraphics[width=1.28in,height=1.28in]{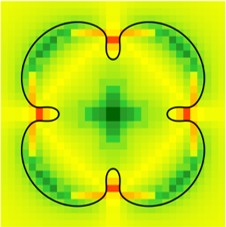}\hfill%
    \includegraphics[width=1.28in,height=1.28in]{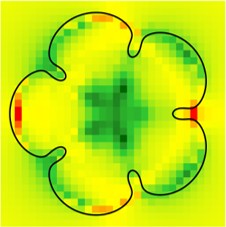}\hfill%
    \includegraphics[width=1.28in,height=1.28in]{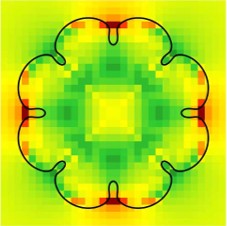}\hfill%
    \includegraphics[width=1.28in,height=1.28in]{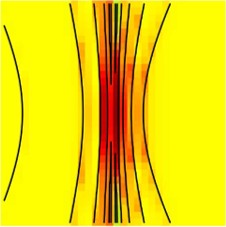}\hfill%
    \includegraphics[width=0.36in,height=1.28in]{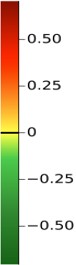}\\
    \parbox{1.28in}{\centering clover-4}\hfill%
    \parbox{1.28in}{\centering clover-5}\hfill%
    \parbox{1.28in}{\centering clover-8}\hfill%
    \parbox{1.28in}{\centering grass}\hfill\parbox{0.36in}{~}
    \caption{Efficacy comparison between $\intboxHermite_4$ and $\intboxTaylor_4$.}
    \label{fig:clover4568gT4H4}
\end{figure}

\section{Conclusion}


	This paper initiates the study of 
	range functions with superior convergence order for bivariate functions.  
	Our experimental study shows the practicality of these
	functions for cubic and quartic convergence order.
    Similar to the univariate setting \citep{Hormann:2021:NRF,Hormann:2023:RFO},
    the sweet spot, that is, the best trade-off between efficiency and efficacy, and ease of implementation, seems to be the cubic convergence order, since fourth order methods give only marginally tighter ranges, are generally slower, and more complicated to implement (see \ref{sec:range-cubic-bivariate} and \ref{sec:range-bicubic}). 
    
	We conclude with some directions for future work:
	\begin{itemize}
	\item
	Clearly, the ideas presented in this paper extend to higher convergence orders
	and to multivariate functions. In particular, Lagrange interpolation at the nodes of a regular $m\times m$ grid with $m>3$ can be used to define recursive Lagrange forms with order $m$ convergence, and similarly for Hermite interpolation of higher order partial derivatives at the corners of the box.    
    However, the
	practicality of such extensions may not be obvious
	and ought to be explored.
	\item
	This paper reports the raw performance of our range functions
	on individual boxes.  A more holistic evaluation
	would be to compare them in the context of an application. 
		As in \citet{Hormann:2021:NRF,Hormann:2023:RFO},
	our Lagrange and Hermite forms are expected to
	outshine Taylor forms in subdivision applications
	such as curve tracing \citep{Plantinga:2004:IAO,cxy,cxyz}.

	We plan to write follow-up papers, in which we will also include details
	about our SLP implementation.
    Moreover, we will extend our framework so as to handle more general functions beyond polynomials.
	\item
	The concept of levels may also be exploited in applications.
	For the root isolation application,
	\citet{Hormann:2023:RFO}
	showed empirically that for several classes of polynomials,
	for a given convergence order $m$, there is an optimal
	level $n$ ($m\le n\le d$) with optimal speedup.
	\end{itemize}

\phantomsection 
\addcontentsline{toc}{section}{\refname} 
\bibliography{references} 

@string{ascg={ACM SIGGRAPH Computer Graphics}}

@string{tog ={ACM Transactions on Graphics}}

@string{cagd={Computer Aided Geometric Design}}

@string{comp={Computing}}

@string{dcg ={Discrete \& Computational Geometry}}

@string{jsc ={Journal of Symbolic Computation}}

@string{tfq ={The Fibonacci Quaterly}}

@string{issac12={Proceedings of the 37th International Symposium on Symbolic and Algebraic Computation}}

@string{issac21={Proceedings of the 2021 ACM International Symposium on Symbolic and Algebraic Computation}}

@string{lncs="Lecture Notes in Computer Science"}

@book{griewank-walther:bk
	, author="Andreas Griewank and Andrea Walther"
	, title="Evaluating Derivatives:
		Principles And Techniques Of Algorithmic Differentiation"
	, edition="2nd"
	, series="Frontiers in Applied Mathematics"
	, publisher="SIAM"
	, remark="There is a version by Griewank alone in 2000"
	, year=2008}

@book{farin:bk
	, title="Curves and Surfaces for Computer
	      Aided Geometric Design"
	, author="Gerald Farin"
	, edition="Third"
	, publisher="Academic Press, Inc"
	, address="San Diego"
	, year=1993 }

@Article{Alladi:1977:OTN,
  author    = {Alladi, Krishnaswami and
               Hoggatt, Jr., V. E.},
  title     = {On {T}ribonacci numbers and related functions},
  journal   = tfq,
  volume    = 15,
  number    = 1,
  month     = feb,
  year      = 1977,
  pages     = {42-45},
  doi       = {10.1080/00150517.1977.12430503}
}

@Article{Cheng:2023:CNR,
  author    = {Cheng, Jin-San and 
               Wen, Junyi and 
               Zhang, Bingwei},
  title     = {Certified numerical real root isolation for bivariate nonlinear systems},
  journal   = {Journal of Symbolic Computation},
  volume    = 114,
  month     = jan # {--} # feb,
  year      = 2023,
  pages     = {149-171},
  doi       = {10.1016/j.jsc.2022.04.005},
}

@Article{Becker:2018:ANS,
  author    = {Becker, Ruben and
               Sagraloff, Michael and
               Sharma, Vikram and
               Yap, Chee},
  title     = {A near-optimal subdivision algorithm for complex root
  	isolation based on {P}ellet test and {N}ewton iteration},
  journal   = jsc,
  volume    = 86,
  month     = may # {--} # jun,
  year      = 2018,
  pages     = {51-96},
  doi       = {10.1016/j.jsc.2017.03.009}
}

@InProceedings{Bowyer:2000:IMI,
  author    = {Bowyer, Adrian and
               Berchtold, Jakob and
               Eisenthal, David and
               Voiculescu, Irina and
               Wise, Kevin},
  title     = {Interval methods in geometric modeling},
  booktitle = {Proceedings of Geometric Modeling and Processing 2000},
  editor    = {Martin, Ralph and Wang, Wenping},
  month     = apr,
  year      = 2000,
  pages     = {321-327},
  address   = {Hong Kong},
  publisher = {IEEE Computer Society},
  doi       = {10.1109/GMAP.2000.838263}
}

@Article{Cornelius:1984:CTR,
  author    = {Cornelius, Herbert and
               Lohner, Rudolf},
  title     = {Computing the range of values of real functions with accuracy higher than second order},
  journal   = comp,
  volume    = 33,
  number    = {3--4},
  month     = sep,
  year      = 1984,
  pages     = {331-347},
  doi       = {10.1007/BF02242276}
}

@Book{Courant:1989:ITC,
  author    = {Courant, Richard and
               John, Fritz},
  title     = {Introduction to Calculus and Analysis},
  volume    = {II},
  year      = 1989,
  publisher = {Springer},
  address   = {New York},
  doi       = {10.1007/978-1-4613-8958-3},
  isbn      = {978-1-4613-8960-6}
}

@Article{Duff:1992:IAA,
  author    = {Duff, Tom},
  title     = {Interval arithmetic recursive subdivision for implicit functions and constructive solid geometry},
  journal   = ascg,
  volume    = 26,
  number    = 2,
  month     = jul,
  year      = 1992,
  pages     = {131-138},
  doi       = {10.1145/142920.134027}
}

@InProceedings{Hormann:2021:NRF,
  author    = {Hormann, Kai and
               Kania, Lucas and
               Yap, Chee},
  title     = {Novel range functions via {T}aylor expansions and
		recursive {L}agrange interpolation with application to real root
  		isolation},
  booktitle = issac21,
  series    = {ISSAC'21},
  month     = jul,
  year      = 2021,
  pages     = {193-200},
  location  = {Saint Petersburg, Russia},
  publisher = {ACM},
  address   = {New York},
  doi       = {10.1145/3452143.3465532}
}

@InCollection{Hormann:2023:RFO,
  author    = {Hormann, Kai and
               Yap, Chee and
               Zhang, Ya Shi},
  title     = {Range functions of any convergence order and their
  	amortized complexity analysis},
  booktitle = {Computer Algebra in Scientific Computing},
  editor    = {Boulier, Fran\c{c}ois and
               England, Matthew and
               Kotsireas, Ilias and
               Sadykov, Timur M. and
               Vorozhtsov, Evgenii V.},
  series    = lncs,
  volume    = 14139,
  year      = 2023,
  pages     = {162-182},
  publisher = {Springer},
  address   = {Cham},
  doi       = {10.1007/978-3-031-41724-5_9}
}

@Article{Martin:2002:COI,
  author    = {Martin, Ralph and
               Shou, Huahao and
               Voiculescu, Irina and
               Bowyer, Adrian and
               Wang, Guojin},
  title     = {Comparison of interval methods for plotting algebraic
  		curves},
  journal   = cagd,
  volume    = 19,
  number    = 7,
  year      = 2002,
  pages     = {553-587},
  doi       = {10.1016/S0167-8396(02)00146-2}
}

@Article{Moessner:2009:EBF,
  author    = {M{\"o}{\ss}ner, Bernhard and
               Reif, Ulrich},
  title     = {Error bounds for polynomial tensor product interpolation},
  journal   = comp,
  volume    = 86,
  number    = {2--3},
  month     = oct,
  year      = 2009,
  pages     = {185-197},
  doi       = {10.1007/s00607-009-0062-7}
}

@Book{Moore:1979:MAA,
  author    = {Moore, Ramon E.},
  title     = {Methods and Applications of Interval Analysis},
  year      = 1979,
  series    = {Studies in Applied and Numerical Mathematics},
  number    = 2,
  publisher = {Society for Industrial and Applied Mathematics},
  address   = {Philadelphia},
  doi       = {10.1137/1.9781611970906},
  isbn      = {978-0-89871-161-5}
}

@InProceedings{Plantinga:2004:IAO,
  author    = {Plantinga, Simon and
               Vegter, Gert},
  title     = {Isotopic approximation of implicit curves and surfaces},
  booktitle = {Proceedings of the 2004 Eurographics/ACM SIGGRAPH Symposium on Geometry Processing},
  series    = {SGP'04},
  year      = 2004,
  pages     = {245-254},
  publisher = {ACM},
  address   = {New York},
  doi       = {10.1145/1057432.1057465}
}

@Book{Ratschek:1984:CMF,
  author    = {Ratschek, Helmut and
               Rokne, Jon},
  title     = {Computer Methods for the Range of Functions},
  series    = {Ellis Horwood Series in Mathematics and its Applications},
  year      = 1984,
  publisher = {Ellis Horwood Limited},
  address   = {Chichester},
  isbn      = {978-0-85312-703-1}
}

@InProceedings{Sharma:2012:NOT,
  author    = {Sharma, Vikram and
               Yap, Chee K.},
  title     = {Near optimal tree size bounds on a simple real root
  	isolation algorithm},
  booktitle = issac12,
  series    = {ISSAC'12},
  month     = jul,
  year      = 2012,
  pages     = {319-326},
  location  = {Grenoble, France},
  publisher = {ACM},
  address   = {New York},
  doi       = {10.1145/2442829.2442875}
}

@Article{Sichetti:2025:MAD,
  author    = {Sichetti, Federico and
               Puppo, Enrico and
               Huang, Zizhou and
               Attene, Marco and
               Zorin, Denis and
               Panozzo, Daniele},
  title     = {Mi{S}o: A {DSL} for Robust and Efficient Solve and Minimize Problems},
  journal   = tog,
  volume    = 44,
  number    = 4,
  month     = jul,
  year      = 2025,
  pages     = {120:1--120:18},
  doi = {10.1145/3731207},
}

@Article{Snyder:1992:IAF,
  author    = {Snyder, John M.},
  title     = {Interval analysis for computer graphics},
  journal   = ascg,
  volume    = 26,
  number    = 2,
  month     = jul,
  year      = 1992,
  pages     = {121-130},
  doi       = {10.1145/142920.134024}
}

@article{cxy
	, crossref="lin-yap:cxy:11"}

@article{lin-yap:cxy:11
    , title="Adaptive Isotopic Approximation of Nonsingular Curves: the
    	Parameterizability and Nonlocal Isotopy Approach"
    , author="Long Lin and Chee Yap"
    , journal= dcg
    , volume=45
    , number=4
    , pages="760--795"
    , comment="DOI: 10.1007/s00454-011-9345-9.	Grant: NSF Grant *977 "
    , year=2011 }

@article{cxyz
	, crossref="lin-yap-yu:cxyz:12"}

@article{lin-yap-yu:cxyz:12
    , title="Non-Local Isotopic Approximation of Nonsingular Surfaces"
    , author="Long Lin and Chee Yap and Jihun Yu"
    , journal="Computer-Aided Design"
    , month=oct
    , volume=45
    , number=2
    , pages="451-462"
    , note="Symp.~on Solid and Physical Modeling (SPM).
		U.~of Burgundy, Dijon, France, Oct 29-31, 2012."
    , year=2012 }

@misc{core:home-range
		, key="CoreLib Home"
		, title="{Range Function Project in CoreLib}"
		, note="Download of source, documentation and data:\\
			\url{https://cs.nyu.edu/exact/core\_pages/svn-core.html}.
			The username and password are both `guest'.
			You can navigate to the range function project
			(\texttt{../corelib2/trunk/progs/eval})
			or go directly to
		\url{https://subversive.cims.nyu.edu/exact/corelib2/trunk/progs/eval}.
			"
		, year="since 2019" }

\appendix

\section{Exact ranges of univariate polynomials with low degree}\label{app:A}
	\label{sec:ranges-univariate}

	We recall the following results from \citep{Hormann:2021:NRF}. Let
	$I=[a,b]$ be a real interval with \emph{radius} $r=(b-a)/2$ and
	\emph{midpoint} $m=(a+b)/2$. Given a polynomial $p$ over $I$, let
	\[
	  \alpha_0 \as \min \{ p(a), p(b) \}, \qquad
	  \beta_0  \as \max \{ p(a), p(b) \}
	\]
	be the minimal and the maximal value of $p$ at the endpoints of $I$.
\subsection{Linear polynomials}\label{sec:range-linear-univariate}
	The exact range of the linear polynomial
	\[
	  p(x) = c_0 + c_1 (x-m)
	\]
	is $p(I)=[\alpha_0,\beta_0]$, because $p$ is monotonic. The range can
	also be expressed as
	\[
	  p(I) = c_0 + r [-1,1] c_1.
	\]
\subsection{Quadratic polynomials}\label{sec:range-quadratic-univariate}
	To determine the exact range $p(I)=[\alpha,\beta]$ of the quadratic
	polynomial
	\[
	  p(x) = c_0 + c_1 (x-m) + c_2 {(x-m)}^2,
	\]
	we first set $\alpha\as\alpha_0$, $\beta\as\beta_0$. Next, we observe
	that the extremum of $p$ occurs at
	\[
	  x^\ast = m - \frac{c_1}{2c_2},
	\]
	which is inside $I$, if and only if $\abs{c_1} < 2\abs{c_2}r$. If
	$x^\ast\notin I$, then $p$ is monotonic on $I$ and our initial
	assignment of $\alpha$ and $\beta$ gives the correct range.
	Otherwise, we check the sign of $c_2$ to see if $p$ has a minimum
	($c_2>0$) or a maximum ($c_2<0$) at $x^\ast$ and accordingly replace
	$\alpha$ or $\beta$ with
	\[
	  p(x^\ast) = c_0 - \frac{c_1^2}{4c_2}.
	\]
	Note that the special case of $p$ being linear (or constant) with
	$c_2=0$ is handled correctly by this procedure.
\subsection{Cubic polynomials}\label{sec:range-cubic-univariate}
	To find the exact range $p(I)=[\alpha,\beta]$ of the cubic polynomial
	\[
	  p(x) = c_0 + c_1 (x-m) + c_2 {(x-m)}^2 + c_3 {(x-m)}^3,
	\]
	we distinguish between the case $c_3=0$, in which we use the procedure in
	\ref{sec:range-quadratic-univariate}, and the case $c_3\ne0$, in which we first set $\alpha\as\alpha_0$ and $\beta\as\beta_0$. If the discriminant
	$\Delta=c_2^2-c_1c_3$ of the quadratic equation $p'(x)=0$ is
	non-positive, then $p$ is monotonic on $I$ and the initial assignment
	of $\alpha$ and $\beta$ gives the correct range. Otherwise, $p$ has a
	local maximum at $x^+$ and a local minimum at $x^-$, where
	\[
	  x^{\pm} = m - \frac{c_2 \pm \sqrt{\Delta}}{3 c_3}.
	\]
	It remains to check if these stationary points lie inside $I$ and to
	adjust $\alpha$ and $\beta$ accordingly.
\section{Exact ranges of bivariate polynomials with low degree}\label{app:B}
	\label{sec:ranges-bivariate}
	
	Let $B=[a_x,b_x]\times[a_y,b_y]$ be a real, rectangular \emph{box}
	with \emph{midpoint} $(m_x,m_y)=(a_x+b_x,a_y+b_y)/2$ and \emph{radii}
	$(r_x,r_y)=(b_x-a_x,b_y-a_y)/2$. Given a bivariate polynomial $p$
	over $B$, let
	\begin{align*}
	  \alpha_0 &\as \min \{ p(a_x,a_y), p(a_x,b_y),
	  		p(b_x,a_y), p(b_x,b_y) \},\\
	  \beta_0  &\as \max \{ p(a_x,a_y), p(a_x,b_y),
	  		p(b_x,a_y), p(b_x,b_y) \}
	\end{align*}
	be the minimal and the maximal value of $p$ at the corners of $B$.
\subsection{Linear polynomials}\label{sec:range-linear-bivariate}
	The exact range of the linear polynomial
	\[
	  p(x,y) = c_{0,0} + c_{1,0} (x-m_x) + c_{0,1} (y-m_y)
	\]
	is $p(B)=[\alpha_0,\beta_0]$, because $p$ is monotonic in $x$ and in
	$y$. The range can also be expressed as
	\[
	  p(B) = c_{0,0} + [-1,1] (r_x\abs{c_{1,0}} + r_y\abs{c_{0,1}}),
	\]
	and if $B$ is a \emph{square} with radius $r=r_x=r_y$, then this
	simplifies to
	\[
	  p(B) = c_{0,0} + r [-1,1] (\abs{c_{1,0}} + \abs{c_{0,1}}).
	\]
	
	For example, if $p$ is the linear Taylor polynomial $T_1$ of $f$
	about the midpoint $\bfm=(m_x,m_y)$ of $B$, then the coefficients to
	be used in the procedure above are
	\[
	  c_{0,0} = f(\bfm), \qquad c_{1,0}
	  	= f^{(1,0)}(\bfm), \qquad c_{0,1} = f^{(0,1)}(\bfm).
	\]
\subsection{Quadratic polynomials}\label{sec:range-quadratic-bivariate}
	To determine the exact range $p(B)=[\alpha,\beta]$ of the quadratic
	polynomial
	\begin{align*}
	  p(x,y)
	  = c_{0,0} &+ c_{1,0} (x-m_x) + c_{0,1} (y-m_y)\\
	            &+ c_{2,0} {(x-m_x)}^2 + c_{1,1} (x-m_x)(y-m_y)
				+ c_{0,2} {(y-m_y)}^2,
	\end{align*}
	we initially set $\alpha\as\alpha_0$, $\beta\as\beta_0$ and then
	search for extrema on the boundary of $B$. To this end, we consider
	the four univariate quadratic polynomials
	\begin{alignat*}{2}
	  p_1(x) &= p(x,a_y), &\qquad x &\in I_1 = [a_x,b_x],\\
	  p_2(x) &= p(x,b_y), & x &\in I_2 = [a_x,b_x],\\
	  p_3(y) &= p(a_x,y), & y &\in I_3 = [a_y,b_y],\\
	  p_4(y) &= p(b_x,y), & y &\in I_4 = [a_y,b_y],
	\end{alignat*}
	analyse them as described in \ref{sec:range-quadratic-univariate},
	and update $\alpha$ and $\beta$ accordingly.
	
	To further check for extrema inside $B$, we compute the
	\emph{discriminant} of $p$,
	\[
	  D_p = \begin{vmatrix}
	          p_{xx} & p_{xy}\\
	          p_{yx} & p_{yy}
	        \end{vmatrix}
	      = 4 c_{2,0} c_{0,2} - c_{1,1}^2.
	\]
	If $D_p<0$, then the unique stationary point of $p$ is a saddle
	point, and if $D_p=0$, then $p$ either has no stationary points or
	there exists a line of infinitely many stationary points with a
	common critical value. In both cases we can be sure that $p$ obtains
	its minimal and maximal value on the boundary of $B$, and the exact
	range $p(B)$ is correctly determined by the previous boundary
	analysis.
	
	If $D_p>0$, then $p$ has a unique stationary point at
	\[
	  (x^\ast,y^\ast)
	  = (m_x,m_y) - \frac{(2c_{1,0}c_{0,2} - c_{0,1}c_{1,1},
	  	2c_{0,1}c_{2,0} - c_{1,0}c_{1,1})}{D_p},
	\]
	which is inside $B$, if and only if
	\[
	  \abs{2c_{1,0}c_{0,2} - c_{0,1}c_{1,1}} < D_p r_x
	  \qquad\text{and}\qquad
	  \abs{2c_{0,1}c_{2,0} - c_{1,0}c_{1,1}} < D_p r_y.
	\]
	In this case, we check the sign of $c_{2,0}$ to see if $p$ has a
	minimum ($c_{2,0}>0$) or a maximum ($c_{2,0}<0$) at $(x^\ast,y^\ast)$
	and accordingly replace $\alpha$ or $\beta$ with
	\[
	  p(x^\ast,y^\ast) = c_{0,0} - \frac{c_{1,0}^2c_{0,2}
	  	- c_{1,0}c_{0,1}c_{1,1} + c_{0,1}^2c_{2,0}}{D_p}.
	\]
	Note that the special case of $p$ being linear (or constant) with
	$c_{2,0}=c_{1,1}=c_{0,2}=0$ is handled correctly by this procedure.
	
	For example, if $p$ is the quadratic Taylor polynomial $T_2$ of $f$
	about the midpoint $\bfm=(m_x,m_y)$ of $B$, then the coefficients to
	be used in the procedure above are
	\[
	  \begin{alignedat}{3}
	    c_{0,0} &= f(\bfm), &\qquad c_{1,0}
			&= f^{(1,0)}(\bfm), &\qquad c_{0,1} &= f^{(0,1)}(\bfm),\\
	    c_{2,0} &= \tfrac12 f^{(2,0)}(\bfm), & c_{1,1}
			&= f^{(1,1)}(\bfm), & c_{0,2} &= \tfrac12 f^{(0,2)}(\bfm).
	  \end{alignedat}
	\]
\subsection{Biquadratic polynomials}\label{sec:range-biquadratic}
	The exact range of the biquadratic polynomial
	\begin{align*}
	  p(x,y)
	  = c_{0,0} &+ c_{1,0} (x-m_x) + c_{0,1} (y-m_y)\\
	            &+ c_{2,0} {(x-m_x)}^2 + c_{1,1} (x-m_x)(y-m_y)
					+ c_{0,2} {(y-m_y)}^2\\
	            &+ c_{2,1} {(x-m_x)}^2 (y-m_y) + c_{1,2} (x-m_x)
					{(y-m_y)}^2 + c_{2,2} {(x-m_x)}^2 {(y-m_x)}^2
	\end{align*}
	can only be computed numerically, since determining the stationary
	points of $p$ requires finding the roots of a polynomial of degree
	five. Instead, we propose to split $p$ into the quadratic Taylor
	polynomial $q$ of $p$ about $\bfm$ and the remainder $r=p-q$, that
	is,
	\begin{align*}
	  r(x,y)
	  = c_{2,1} {(x-m_x)}^2 (y-m_y) + c_{1,2} (x-m_x) {(y-m_y)}^2
	  	+ c_{2,2} {(x-m_x)}^2 {(y-m_x)}^2,
	\end{align*}
	and to approximate the range of $p$ by $q(B)+r(B)\supseteq p(B)$.
	While the range of $q$ can be computed as described in
	\ref{sec:range-quadratic-bivariate}, a careful analysis of $r$
	(finding the stationary points, checking the discriminant, etc.)
	reveals that it does not have isolated local extrema. The range of
	$r$ can hence be found by analyzing $r$ over the boundary of $B$, as
	explained in the first step of the procedure in
	\ref{sec:range-quadratic-bivariate}.
	
	For example, if $p$ is the biquadratic Lagrange interpolant of $f$ at
	$\{a_x,m_x,b_x\}\times\{a_y,m_y,b_y\}$, then the coefficients of $p$
	are
	\begin{gather*}
	  c_{0,0} = f_{1,1}, \qquad
	  c_{1,0} = \frac{f_{2,1}-f_{0,1}}{2 r_x}, \qquad
	  c_{0,1} = \frac{f_{1,2}-f_{1,0}}{2 r_y}, \\[\jot]
	  c_{2,0} = \frac{f_{2,1}-2f_{1,1}+f_{0,1}}{2 r_x^2}, \qquad
	  c_{1,1} = \frac{f_{2,2}-f_{0,2}-f_{2,0}+f_{0,0}}{4 r_x r_y}, \qquad
	  c_{0,2} = \frac{f_{1,2}-2f_{1,1}+f_{1,0}}{2 r_y^2}, \\[\jot]
	  c_{2,1} = \frac{f_{2,2}-2f_{1,2}+f_{0,2}-f_{2,0}+2f_{1,0}
	  		-f_{0,0}}{4 r_x^2 r_y}, \qquad
	  c_{1,2} = \frac{f_{2,2}-2f_{2,1}+f_{2,0}-f_{0,2}+2f_{0,1}
	  		-f_{0,0}}{4 r_x r_y^2}\\[\jot]
	  c_{2,2} = \frac{f_{2,2}-2f_{1,2}+f_{0,2}-2f_{2,1}+4f_{1,1}
	  		-2f_{0,1}+f_{2,0}-2f_{1,0}+f_{0,0}}
	                 {4 r_x^2 r_y^2},
	\end{gather*}
	where
	\[
	  f_{i,j} = f(m_x+(i-1)r_x, m_y+(j-1)r_y), \qquad
	  0 \le i,j \le 2.
	\]
	
\subsection{Cubic polynomials}\label{sec:range-cubic-bivariate}
	To determine the exact range $p(B)=[\alpha,\beta]$ of the cubic
	polynomial
	\begin{align*}
	  p(x,y)
	  = c_{0,0} &+ c_{1,0} (x-m_x) + c_{0,1} (y-m_y)\\
	            &+ c_{2,0} {(x-m_x)}^2 + c_{1,1} (x-m_x)(y-m_y)
					+ c_{0,2} {(y-m_y)}^2\\
	            &+ c_{3,0} {(x-m_x)}^3 + c_{2,1} {(x-m_x)}^2 (y-m_y)
	             + c_{1,2} (x-m_x) {(y-m_y)}^2 + c_{0,3} {(y-m_x)}^3,
	\end{align*}
	we initially set $\alpha\as\alpha_0$, $\beta\as\beta_0$ and then
	search for extrema on the boundary of $B$, adjusting $\alpha$ and
	$\beta$ accordingly, like in \ref{sec:range-quadratic-bivariate}, but
	taking into account that the four polynomials are cubic and have to
	be analysed with the procedure in \ref{sec:range-cubic-univariate}.
	To further check for extrema inside $B$, we have to determine the
	common roots of the first order partial derivatives of $p$. To this
	end, we express the latter as polynomials in $x$,
	\[
	  A(x) \as p^{(1,0)}(x+m_x,y+m_y) = a_0 x^2 + a_1 x + a_2, \qquad
	  B(x) \as p^{(0,1)}(x+m_x,y+m_y) = b_0 x^2 + b_1 x + b_2,
	\]
	with coefficients in $y$,
	\begin{align*}
	  a_0 &\as 3 c_{3,0}, &
	  a_1 &\as 2 c_{2,0} + 2 c_{2,1} y, &
	  a_2 &\as c_{1,0} + c_{1,1} y + c_{1,2} y^2,\\
	  b_0 &\as c_{2,1}, &
	  b_1 &\as c_{1,1} + 2 c_{1,2} y, &
	  b_2 &\as c_{0,1} + 2 c_{0,2} y + 3 c_{0,3} y^2,
	\end{align*}
	and consider the determinant of the Sylvester matrix,
	\[
	  \begin{vmatrix}
	    a_0 & a_1 & a_2 &  0  \\
	     0  & a_0 & a_1 & a_2 \\
	    b_0 & b_1 & b_2 &  0  \\
	     0  & b_0 & b_1 & b_2
	  \end{vmatrix}
	  = {(a_0 b_2 - a_2 b_0)}^2 - (a_0 b_1 - a_1 b_0) (a_1 b_2 - a_2 b_1).
	\]
	The roots of this quartic polynomial (in $y$) are the $y$-coordinates
	of the stationary points of $p$, and they can be found, for example,
	with \emph{Ferrari's method}. To find the corresponding
	$x$-coordinates, we need to substitute each of the (up to four) roots
	into the coefficients of $A(x)$ and solve for $x$. While $A$
	certainly vanishes at the resulting (up to eight) pairs of
	coordinates, we must further check that also $B(x)$ evaluates to
	zero, which reduces the number of stationary points of $p$ to at most
	four, as guaranteed by B{\'e}zout's theorem. For those stationary
	points that lie inside $B$, we finally use the second derivative test
	to find out if they are local extrema and update $\alpha$ and $\beta$
	accordingly.
	
	For example, if $p$ is the cubic Taylor polynomial $T_3$ of $f$ about
	the midpoint $\bfm=(m_x,m_y)$ of $B$, then the coefficients to be
	used in the procedure above are
	\begin{gather*}
	  \begin{alignedat}{3}
	    c_{0,0} &= f(\bfm), &\qquad c_{1,0}
			&= f^{(1,0)}(\bfm), &\qquad c_{0,1} &= f^{(0,1)}(\bfm),\\
	    c_{2,0} &= \tfrac12 f^{(2,0)}(\bfm), & c_{1,1}
			&= f^{(1,1)}(\bfm), & c_{0,2} &= \tfrac12 f^{(0,2)}(\bfm).
	  \end{alignedat}\\
	  c_{3,0} = \tfrac16 f^{(3,0)}(\bfm), \qquad
	  c_{2,1} = \tfrac12 f^{(2,1)}(\bfm), \qquad
	  c_{1,2} = \tfrac12 f^{(1,2)}(\bfm), \qquad
	  c_{0,3} = \tfrac16 f^{(0,3)}(\bfm).
	\end{gather*}
\subsection{Bicubic polynomials}\label{sec:range-bicubic}
	For a bicubic polynomial
	\begin{align*}
	  p(x,y)
	  = c_{0,0} &+ c_{1,0} (x-m_x) + c_{0,1} (y-m_y)\\
	            &+ c_{2,0} {(x-m_x)}^2 + c_{1,1} (x-m_x)(y-m_y) + c_{0,2} {(y-m_y)}^2\\
	            &+ c_{3,0} {(x-m_x)}^3 + c_{2,1} {(x-m_x)}^2 (y-m_y)
	             + c_{1,2} (x-m_x) {(y-m_y)}^2 + c_{0,3} {(y-m_x)}^3\\
	            &+ c_{3,1} {(x-m_x)}^3 (y-m_y) + c_{2,2} {(x-m_x)}^2 {(y-m_y)}^2
	             + c_{1,3} (x-m_x) {(y-m_x)}^3\\
	            &+ c_{3,2} {(x-m_x)}^3 {(y-m_y)}^2 + c_{2,3} {(x-m_x)}^2 {(y-m_y)}^3
	             + c_{3,3} {(x-m_x)}^3 {(y-m_y)}^3,
	\end{align*}
	we proceed as for biquadratic polynomials in
	\ref{sec:range-biquadratic}. We split $p$ into the cubic Taylor
	polynomial $q$ of $p$ about $\bfm$ and the remainder $r=p-q$, that
	is,
	\begin{align*}
	  r(x,y) &= c_{3,1} {(x-m_x)}^3 (y-m_y) + c_{2,2} {(x-m_x)}^2 {(y-m_y)}^2
	          + c_{1,3} (x-m_x) {(y-m_x)}^3\\
	         &+ c_{3,2} {(x-m_x)}^3 {(y-m_y)}^2 + c_{2,3} {(x-m_x)}^2 {(y-m_y)}^3
	          + c_{3,3} {(x-m_x)}^3 {(y-m_y)}^3,
	\end{align*}
	and approximate the range of $p$ by $q(B)+r(B)\supseteq p(B)$.
	While the range of $q$ can be computed as described in
	\ref{sec:range-cubic-bivariate}, a slight variation of this procedure
	can be used to determine the range of $r$. The restriction of $r$ to
	the boundary of $B$ gives four cubic polynomials, whose extrema are
	determined in the same way. To find the extrema of $r$ inside $B$, we
	observe that the first order partial derivatives of $r$ are
	\[
	  r^{(1,0)}(x+m_x,y+m_y) = y A(x), \qquad
	  r^{(0,1)}(x+m_x,y+m_y) = x B(x),
	\]
	where
	\[
	  A(x) \as a_0 x^2 + a_1 x + a_2, \qquad
	  B(x) \as b_0 x^2 + b_1 x + b_2
	\]
	are polynomials in $x$ with coefficients in $y$,
	\begin{align*}
	  a_0 &\as 3 c_{3,1} + 3 c_{3,2} y + 3 c_{3,3} y^2, &
	  a_1 &\as 2 c_{2,2} y + 2 c_{2,3} y^2, &
	  a_2 &\as c_{1,3} y^2,\\
	  b_0 &\as c_{3,1} + 2 c_{3,2} y + 3 c_{3,3} y^2, &
	  b_1 &\as 2 c_{2,2} y + 3 c_{2,3} y^2, &
	  b_2 &\as 3 c_{1,3} y^2.
	\end{align*}
	The determinant of the corresponding Sylvester matrix is $y^4 D(y)$,
	where $D$ is a quartic polynomial, leading to at most four stationary
	points of $r$, which need to be further checked. The additional
	stationary point of $r$ at $(m_x,m_y)$ can be ignored, because
	$r(m_x,m_y)=0$, and this value is already contained in the ranges of
	the boundary polynomials, because $r(x,m_y)=r(m_x,y)=0$ for any
	$x,y\in\RR$.
	
	
	For example, if $p$ is the bicubic Hermite interpolant of $f$ at the
	corners of $B$, then the coefficients of $p$ are

	\newcommand{\mycoeff}[5]{\tfrac{#2 f^{#1}_{0,0}
		#3 f^{#1}_{0,1} #4 f^{#1}_{1,0} #5 f^{#1}_{1,1}}}
	\begin{alignat*}{4}
	  c_{0,0} &=    \mycoeff{  }++++{4}\,
	          &+\,& \mycoeff{x }++--{8 r_x^{-1}}\,
	          &+\,& \mycoeff{y }+-+-{8 r_y^{-1}}\,
	          &+\,& \mycoeff{xy}+--+{16 r_x^{-1} r_y^{-1}},\\
	  c_{1,0} &=    \mycoeff{  }--++{8 r_x/3}
	          &+\,& \mycoeff{x }----{8}
	          &+\,& \mycoeff{y }-++-{16 r_x r_y^{-1}/3}
	          &+\,& \mycoeff{xy}-+-+{16 r_y^{-1}},\\
	  c_{0,1} &=    \mycoeff{  }-+-+{8 r_y/3}
	          &+\,& \mycoeff{x }-++-{16 r_y r_x^{-1}/3}
	          &+\,& \mycoeff{y }----{8}
	          &+\,& \mycoeff{xy}--++{16 r_x^{-1}},\\
	  c_{2,0} &=&&  \mycoeff{x }--++{8 r_x}&&
	          &+\,& \mycoeff{xy}-++-{16 r_x r_y^{-1}},\\
	  c_{1,1} &=    \mycoeff{  }+--+{16 r_x r_y/9}
	          &+\,& \mycoeff{x }+-+-{16 r_y/3}
	          &+\,& \mycoeff{y }++--{16 r_x/3}
	          &+\,& \mycoeff{xy}++++{16 r_x^{-1}},\\
	  c_{0,2} &=&&&&\mycoeff{y }-+-+{8 r_y}
	          &+\,& \mycoeff{xy}-++-{16 r_y r_x^{-1}},\\
	  c_{3,0} &=    \mycoeff{  }++--{8 r_x^3}
	          &+\,& \mycoeff{x }++++{8 r_x^2}
	          &+\,& \mycoeff{y }+--+{16 r_x^3 r_y^{-1}}
	          &+\,& \mycoeff{xy}+-+-{16 r_x^2 r_y^{-1}},\\
	  c_{2,1} &=&&  \mycoeff{x }+--+{16 r_x r_y/3}&&
	          &+\,& \mycoeff{xy}++--{16 r_x},\\
	  c_{1,2} &=&&&&\mycoeff{y }+--+{16 r_x r_y/3}
	          &+\,& \mycoeff{xy}+-+-{16 r_y},\\
	  c_{0,3} &=    \mycoeff{  }+-+-{8 r_y^3}
	          &+\,& \mycoeff{x }+--+{16 r_y^3 r_x^{-1}}
	          &+\,& \mycoeff{y }++++{8 r_y^2}
	          &+\,& \mycoeff{xy}++--{16 r_y^2 r_x^{-1}},\\
	  c_{3,1} &=    \mycoeff{  }-++-{16 r_x^3 r_y/3}
	          &+\,& \mycoeff{x }-+-+{16 r_x^2 r_y/3}
	          &+\,& \mycoeff{y }--++{16 r_x^3}
	          &+\,& \mycoeff{xy}----{16 r_x^2},\\
	  c_{2,2} &=&&&&&&\mycoeff{xy}+--+{16 r_x r_y},\\
	  c_{1,3} &=    \mycoeff{  }-++-{16 r_x r_y^3/3}
	          &+\,& \mycoeff{x }-+-+{16 r_y^3}
	          &+\,& \mycoeff{y }--++{16 r_x r_y^2/3}
	          &+\,& \mycoeff{xy}----{16 r_y^2},\\
	  c_{3,2} &=&&&&\mycoeff{y }-++-{16 r_x^3 r_y}
	          &+\,& \mycoeff{xy}-+-+{16 r_x^2 r_y},\\
	  c_{2,3} &=&&  \mycoeff{x }-++-{16 r_x r_y^3}&&
	          &+\,& \mycoeff{xy}--++{16 r_x r_y^2},\\
	  c_{3,3} &=    \mycoeff{  }+--+{16 r_x^3 r_y^3}
	          &+\,& \mycoeff{x }+-+-{16 r_x^2 r_y^3}
	          &+\,& \mycoeff{y }++--{16 r_x^3 r_y^2}
	          &+\,& \mycoeff{xy}++++{16 r_x^2 r_y^2},
	\end{alignat*}
	where
	\[
	  \begin{aligned}
	       f_{i,j} &= \phantom{{}^{(0,0)}}
		   		f(m_x+(2i-1)r_x, m_y+(2j-1)r_y),\\
	       f^x_{i,j} &= f^{(1,0)}(m_x+(2i-1)r_x, m_y+(2j-1)r_y),\\
	       f^y_{i,j} &= f^{(0,1)}(m_x+(2i-1)r_x, m_y+(2j-1)r_y),\\
	    f^{xy}_{i,j} &= f^{(1,1)}(m_x+(2i-1)r_x, m_y+(2j-1)r_y),
	  \end{aligned}
	\]
	for $0 \le i,j \le 1$.
\section{Additional Experiments}
\begin{table}\small\centering
  \begin{tabular}{lcl}
    \toprule
    name & degree & $f(x,y)$\\
    \midrule
    cardioid
    & $4$
    & \scalebox{1}{\parbox{245pt}{
      $\displaystyle (x^2+y^2+x)^2-(x^2+y^2)$}}\\

    lemniscate
    & $4$
    & \scalebox{1}{\parbox{245pt}{
      $\displaystyle (x^2+y^2)^2-2(x^2-y^2)$}}\\

    octic-flower
    & $8$
    & \scalebox{0.95}{\parbox{290pt}{
      $2000y^8 + 8000x^2y^6 + 12000x^4y^4 + 8000x^6y^2 + 2000x^8$\\
      $- 3000y^6 + 9000x^2y^4 - 21000x^4y^2 - 1000x^6 + 1$}}\\
    \bottomrule
  \end{tabular}
  \caption{Three additional polynomial test functions used in our numerical experiments.}
  \label{tab:test-functions-new}
\end{table}
In this section, we provide additional experimental results to further support the conclusions of the main text. Specifically, we consider the following three additional bivariate polynomial examples:

In Table~\ref{tab:LT-new}, we evaluate the five range functions of~\eqref{eq:5ranges} on the set of $1024$ boxes obtained by uniformly subdividing the corresponding square domain for each test function into a $32\times 32$ grid. As in Table~\ref{tab:LT}, we take $\intboxTaylor_2$ as the baseline method, define the speedup ratio as the ratio of the total running time of $\intboxTaylor_2$ to that of the method under consideration, and define the efficacy ratio as the ratio of the total output width of $\intboxTaylor_2$ to that of the corresponding method. Thus, a ratio greater than $1$ indicates an improvement over the baseline. The column ``time (ms)'' reports the average running time over 10 runs, measured using Julia's benchmarking tool. As before, for $\intboxLagrange_3$ and $\intboxHermite_4$ we also report the versions with shared data, since in subdivision algorithms these methods can reuse function values and derivative information across neighboring boxes, leading to improved amortized time and memory costs.

\begin{table}
\centering\small
\begin{tabular}{llcccc}
\toprule
test function & range function & time (ms) & speedup & efficacy & memory (MB) \\
\midrule
\multirow{7}{*}{cardioid}
& $\intboxTaylor_2$ (baseline) & \z68.58 & 1 & 1 & \z26.67 \\\cmidrule(lr){2-6}
& $\intboxTaylor_3$ & \z63.96 & 1.07 & 1.0710 & \z25.86 \\
& $\intboxTaylor_4$ & \z83.85 & 0.82 & 1.0712 & \z40.89 \\\cmidrule(lr){2-6}
& $\intboxLagrange_3$
& \z99.30 & 0.69 & \multirow{2}{*}{1.0703} & \z29.70 \\
& $\intboxLagrange_3$ (shared)
& \color{mygreen} \z35.61 & \color{mygreen} 1.93 & & \color{mygreen} \z9.73 \\\cmidrule(lr){2-6}
& $\intboxHermite_4$
& \z134.24 & 0.51 & \multirow{2}{*}{\color{mygreen} 1.0713} & \z84.61 \\
& $\intboxHermite_4$ (shared)
& \z55.50 & 1.24 & & \z53.09 \\
\midrule
\multirow{7}{*}{lemniscate}
& $\intboxTaylor_2$ (baseline) & \z72.83 & 1 & 1 & \z27.14 \\\cmidrule(lr){2-6}
& $\intboxTaylor_3$ & \z53.79 & 1.35 & 1.0671 & \z26.33 \\
& $\intboxTaylor_4$ & \z75.38 & 0.97 & 1.0676 & \z41.36 \\\cmidrule(lr){2-6}
& $\intboxLagrange_3$
& \z94.85 & 0.77 & \multirow{2}{*}{1.0669} & \z30.54 \\
& $\intboxLagrange_3$ (shared)
& \color{mygreen} \z35.93 & \color{mygreen} 2.03 & & \color{mygreen} \z10.12 \\\cmidrule(lr){2-6}
& $\intboxHermite_4$
& \z99.52 & 0.73 & \multirow{2}{*}{\color{mygreen} 1.0676} & \z85.04 \\
& $\intboxHermite_4$ (shared)
& \z56.72 & 1.28 & & \z53.20 \\
\midrule
\multirow{7}{*}{octic-flower}
& $\intboxTaylor_2$ (baseline) & \z277.22 & 1 & 1 & \z101.49 \\\cmidrule(lr){2-6}
& $\intboxTaylor_3$ & \z261.19 & 1.06 & 1.1581 & \z100.67 \\
& $\intboxTaylor_4$ & \z321.33 & 0.86 & 1.1604 & \z117.35 \\\cmidrule(lr){2-6}
& $\intboxLagrange_3$
& \z361.44 & 0.77 & \multirow{2}{*}{1.1562} & \z94.53 \\
& $\intboxLagrange_3$ (shared)
& \color{mygreen} \z140.09 & \color{mygreen} 1.98 & & \color{mygreen} \z37.61 \\\cmidrule(lr){2-6}
& $\intboxHermite_4$
& \z826.92 & 0.34 & \multirow{2}{*}{\color{mygreen} 1.1606} & \z523.77 \\
& $\intboxHermite_4$ (shared)
& \z511.36 & 0.54 & & \z410.90 \\
\bottomrule
\end{tabular}
\caption{Performance comparison of the range functions in~\eqref{eq:5ranges} for the 1024 boxes obtained by uniformly subdividing the corresponding square domain for each test function. The domains are $[-2.0,2.0]^2$ for cardioid, $[-1.5,1.5]^2$ for lemniscate, and $[-1.2,1.2]^2$ for octic-flower. For each test function, we highlight in {\color{mygreen}green} the best values under \emph{time}, \emph{speedup}, \emph{efficacy}, and \emph{memory} usage.}
\label{tab:LT-new}
\end{table}

From the data in Table~\ref{tab:LT-new}, we again observe that the efficacy (tightness) is greater than $1$ for all methods of higher convergence order, as expected. In this additional set of examples, the efficacy ratios lie in the range $[1.0669,1.1606]$. For $\intboxTaylor_3$, this improvement in tightness is achieved without sacrificing efficiency: its speedup ratios range from $1.06$ to $1.35$, and its memory usage remains very close to that of the baseline $\intboxTaylor_2$. For $\intboxTaylor_4$, the efficacy is slightly better than that of $\intboxTaylor_3$, but the gain is marginal, while the running time and memory usage are somewhat worse. Thus, as in the main text, the improvement from order $m=3$ to $m=4$ appears to be relatively modest.

By contrast, for $\intboxLagrange_3$ and $\intboxHermite_4$ without sharing, the tradeoff between tightness and efficiency is again clearly visible. Their efficacy is competitive and often slightly better than that of the Taylor forms of the same order, but this comes at a noticeable runtime cost, especially for $\intboxHermite_4$, whose speedup drops to $0.34$ on the octic-flower example and whose memory usage becomes substantially larger. However, once data sharing is enabled, $\intboxLagrange_3$ becomes the most efficient method overall in these experiments: its speedup ratios are close to $2$ for all three examples, while its memory usage is reduced dramatically. The shared version of $\intboxHermite_4$ also benefits substantially from reuse, although for the more complicated octic-flower example it still remains slower and considerably more memory-intensive than the other methods.

Therefore, these additional experiments lead to the same qualitative conclusions as those in Table~\ref{tab:LT}: 
(1) the efficacy generally improves with the convergence order, but the improvement from $m=3$ to $m=4$ is limited; 
(2) among the methods without sharing, $\intboxTaylor_3$ offers the best balance between tighter enclosures and computational efficiency; and 
(3) with data sharing, $\intboxLagrange_3$ is again the most efficient method in terms of both runtime and memory usage.

Next, we present the pairwise efficacy comparison plots for these additional examples. As in the main text, for each box in the corresponding $32\times 32$ grid, the color encodes the logarithm of the ratio of the widths produced by two competing range functions. Green hues indicate boxes where the first method is tighter, red hues indicate boxes where the second method is tighter, and yellow hues indicate comparable performance.

These plots are consistent with the conclusions suggested by Table~\ref{tab:LT-new}. Figure~\ref{fig:ExtraT32} shows that $\intboxTaylor_3$ is generally tighter than the baseline $\intboxTaylor_2$. Figure~\ref{fig:ExtraT43} further indicates that $\intboxTaylor_4$ is in general tighter than $\intboxTaylor_3$, although the improvement is again modest on average. Figures~\ref{fig:ExtraL3T3} and~\ref{fig:ExtraH4T4} show that the Lagrange- and Hermite-based forms can outperform the corresponding Taylor forms on some regions while being outperformed on others, so neither method can be declared a uniform pointwise winner over the entire landscape. This is fully consistent with the aggregate efficacy data in Table~\ref{tab:LT-new}, which show only a mild advantage in tightness for the non-Taylor forms, but a much stronger dependence on whether data sharing is available.
\begin{figure}[h]
    \centering

    \begin{subfigure}[t]{0.48\textwidth}
        \centering
        \includegraphics[width=\linewidth]{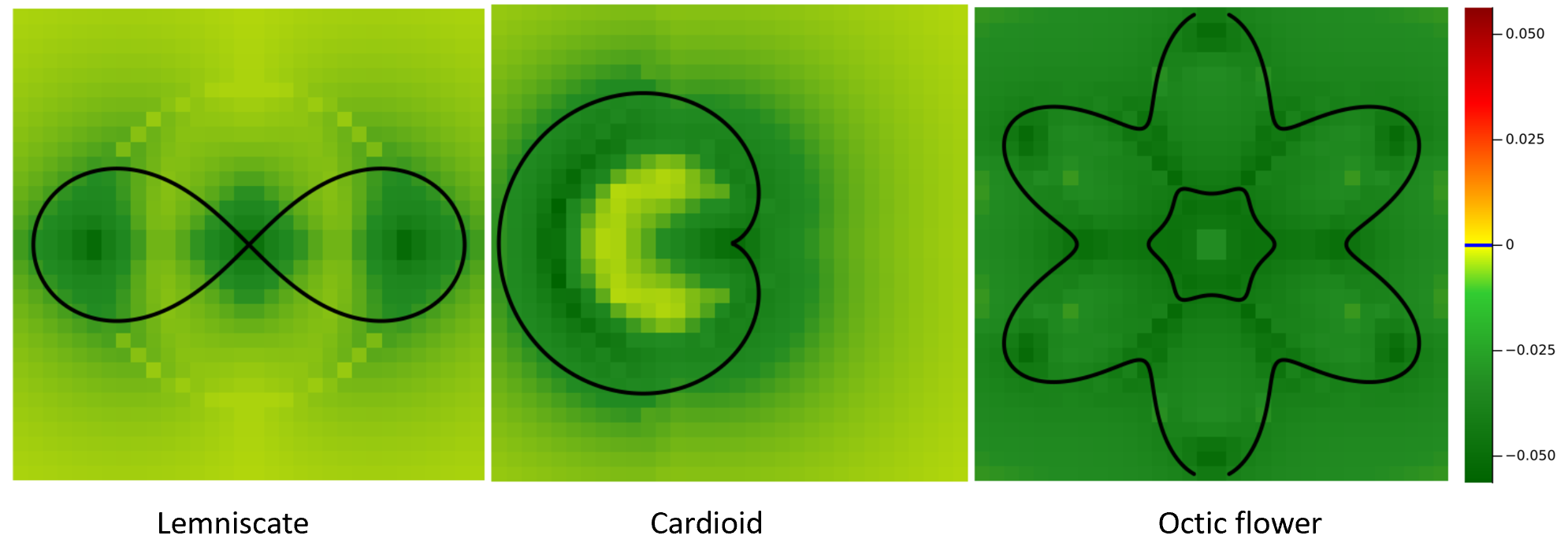}
        \caption{Efficacy comparison between $\intboxTaylor_3$ and $\intboxTaylor_2$.}
        \label{fig:ExtraT32}
    \end{subfigure}
    \hfill
    \begin{subfigure}[t]{0.48\textwidth}
        \centering
        \includegraphics[width=\linewidth]{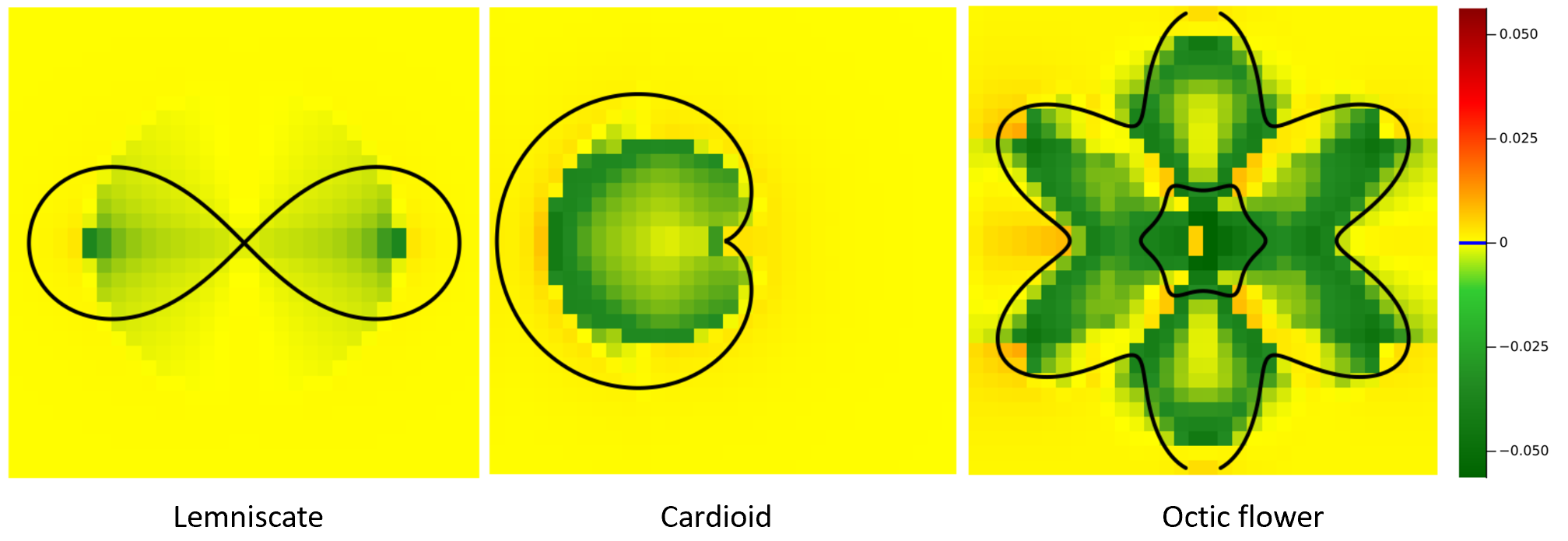}
        \caption{Efficacy comparison between $\intboxLagrange_3$ and $\intboxTaylor_3$.}
        \label{fig:ExtraL3T3}
    \end{subfigure}

    \vspace{0.8em}

    \begin{subfigure}[t]{0.48\textwidth}
        \centering
        \includegraphics[width=\linewidth]{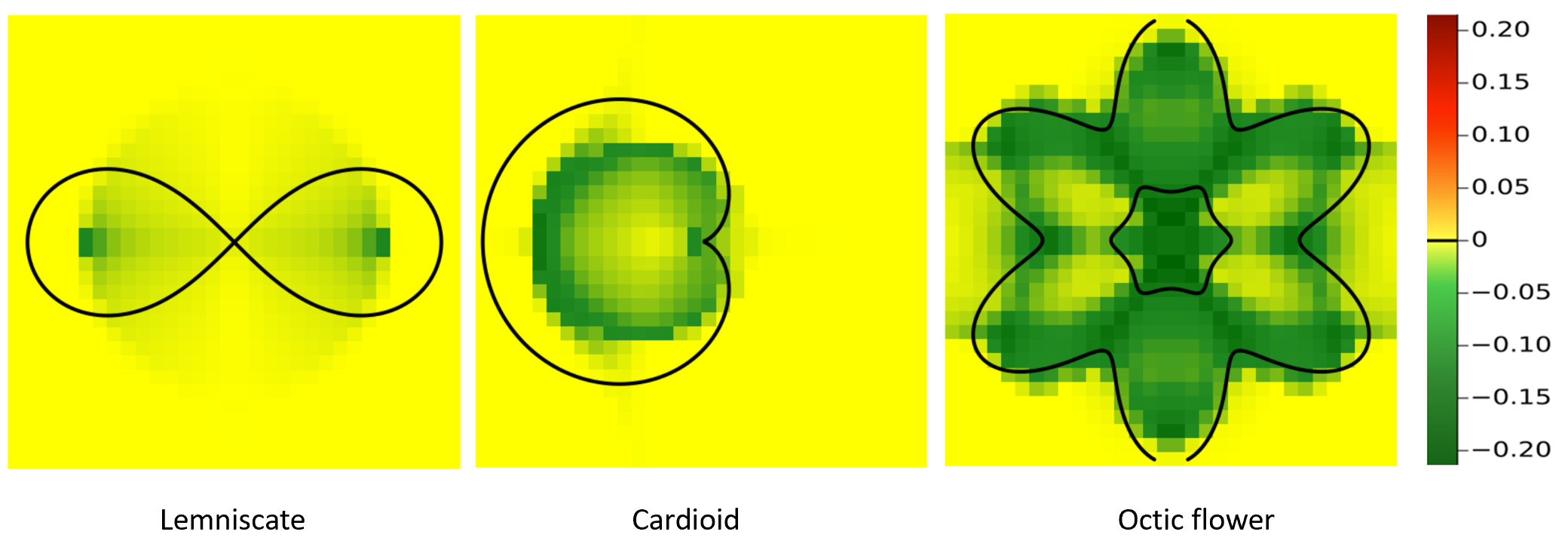}
        \caption{Efficacy comparison between $\intboxTaylor_4$ and $\intboxTaylor_3$.}
        \label{fig:ExtraT43}
    \end{subfigure}
    \hfill
    \begin{subfigure}[t]{0.48\textwidth}
        \centering
        \includegraphics[width=\linewidth]{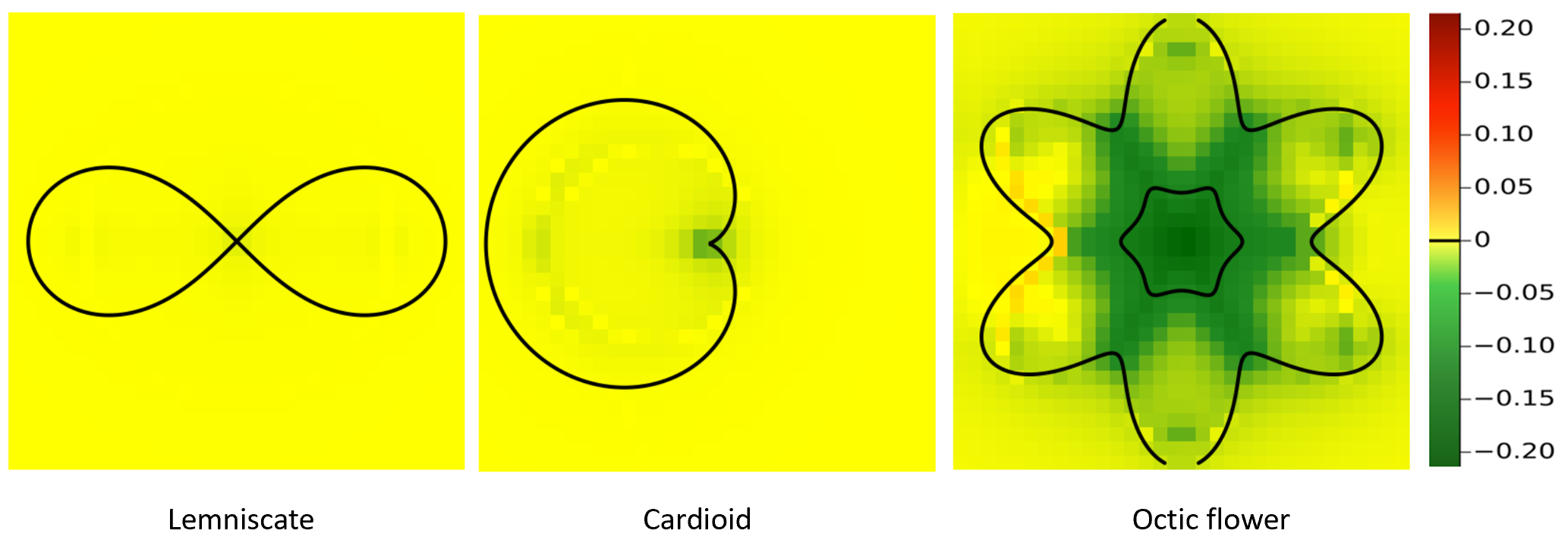}
        \caption{Efficacy comparison between $\intboxHermite_4$ and $\intboxTaylor_4$.}
        \label{fig:ExtraH4T4}
    \end{subfigure}

    \caption{Pairwise efficacy comparisons for the additional examples.}
    \label{fig:ExtraPairwise}
\end{figure}
\end{document}